\documentclass[11pt,reqno]{amsart}

\usepackage[utf8]{inputenc}
\usepackage[T1]{fontenc}
\usepackage{amsmath,amssymb,amsthm,mathtools}
\usepackage{enumitem}
\usepackage[margin=1.15in]{geometry}
\usepackage{xcolor}
\usepackage{hyperref}
\usepackage{tikz}
\usetikzlibrary{positioning, arrows.meta, shapes.geometric, calc, fit, backgrounds}
\hypersetup{
  colorlinks=true,
  linkcolor=blue!50!black,
  citecolor=blue!50!black,
  urlcolor=blue!50!black
}
\usepackage{graphicx}
\usepackage{cite}

\newtheorem{theorem}{Theorem}[section]
\newtheorem{proposition}[theorem]{Proposition}
\newtheorem{lemma}[theorem]{Lemma}
\newtheorem{corollary}[theorem]{Corollary}

\theoremstyle{definition}

\newtheorem{problem}[theorem]{Problem}
\newtheorem{conjecture}[theorem]{Conjecture}

\theoremstyle{remark}
\newtheorem{remark}[theorem]{Remark}

\newcommand{\R}{\mathbb{R}}
\newcommand{\T}{\mathbb{T}}
\newcommand{\Z}{\mathbb{Z}}
\newcommand{\Hil}{\mathcal{H}}
\newcommand{\Free}{\mathcal{F}}
\newcommand{\Vop}{\mathcal{V}}
\newcommand{\Adm}{\mathcal{A}}
\newcommand{\loc}{\mathrm{loc}}
\newcommand{\dist}{\operatorname{dist}}
\newcommand{\Div}{\nabla\!\cdot}
\DeclareMathOperator{\Op}{Op}

\title[Critical Mass in Chemotaxis: Elliptic vs Volterra]
{Structural dichotomy and mass criticality in indirect chemotaxis cascades: fourth-order ellipticity versus Volterra memory}

\author{Jiguang Yu}
\address{College of Engineering, Boston University, Boston, MA 02215, USA}
\email{jyu678@bu.edu}
\thanks{These authors (J. Yu and L.S. Wang) contributed equally to this work. Corresponding  author: L.S. Wang (wang.s41@northeastern.edu)}

\author{Louis Shuo Wang}
\address{Department of Mathematics, Northeastern University, Boston, MA 02115, USA}
\email{wang.s41@northeastern.edu}

\date{}

\begin{document}

\keywords{chemotaxis cascade, Keller–Segel system, critical mass, fourth-order elliptic operator, logarithmic Hardy–Littlewood–Sobolev inequality.}

\subjclass[2020]{Primary 35Q92, 35B33; Secondary 92C17, 35R09, 35A15, 35K51}

\begin{abstract}
We investigate the structural and operator-theoretic foundations of indirect signal-generation mechanisms in Keller–Segel-type chemotaxis models. By analyzing a physically motivated multi-stage signaling cascade, we establish a precise mathematical dichotomy between instantaneous equilibration and transient kinetic memory. Specifically, we prove that the fully equilibrating parabolic–elliptic–elliptic (PES) cascade reduces to a static fourth-order elliptic interaction. In dimension four, an exact algebraic cancellation of the leading Newtonian singularity yields a purely logarithmic kernel, shifting the mass-critical dimension from $N=2$ to $N=4$. Through an $L^2$-gradient flow formulation, we identify the corresponding concentration-scaling candidate critical mass $M_* = 64\pi^2\tau/\chi$. In sharp contrast, we demonstrate that the mixed elliptic–parabolic (MEP) cascade retains a genuine Volterra memory effect that defies static reduction. Its interaction drift acts as a singular perturbation in time—exhibiting classical two-dimensional Keller–Segel principal order near the time diagonal, yet providing fourth-order smoothing in its frozen-time average, necessitating a mixed space–time threshold theory. These results isolate the physical origin of mass-critical dimensional shifts in multiscale biological systems and formulate the specific adapted Adams/logarithmic Hardy–Littlewood–Sobolev (log-HLS) inequalities and mixed-norm criteria required to close the threshold problems.
\end{abstract}

\maketitle

\section{Introduction}
\label{sec:intro}

\subsection{Background}
\label{subsec:background}
In the mathematical modelling of biological and physical self-organisation, chemotaxis---the directed movement of cellular or synthetic agents in response to chemical gradients---represents a fundamental mechanism for macroscopic pattern formation. Models of Keller--Segel type \cite{keller1970initiation,keller1971model} exhibit critical behaviour through a delicate balance between random diffusion and chemotactic aggregation. In the classical parabolic--elliptic Keller--Segel system
\begin{equation}\label{eq:KS-classical}
        \partial_t u=\Delta u-\chi\Div(u\nabla v),
        \qquad
        -\Delta v=u,
\end{equation}
the eliminated signal law is $v=(-\Delta)^{-1}u$. The corresponding drift operator $\nabla(-\Delta)^{-1}$ has principal pseudodifferential order $-1$, and under the parabolic scaling
$u_\lambda(x,t)=\lambda^{\alpha}u(\lambda x,\lambda^{2}t)$, diffusion and chemotactic aggregation balance precisely for $\alpha=2$. Hence, the mass-critical dimension is $N=2$ and the scaling-critical Lebesgue exponent is $q_c=N/2$. In two dimensions, the logarithmic singularity of the Newtonian potential and the logarithmic Hardy--Littlewood--Sobolev (log-HLS) inequality \cite{beckner1993sharp,carlen1992competing} underpin the sharp critical-mass theorem $M_c=8\pi/\chi$ of Blanchet--Dolbeault--Perthame \cite{blanchet2006two}; see also \cite{jager1992explosions, nagai1995blow, gajewski1998global,liu2025bidirectional,horstmann20031970}.

In realistic biochemical networks, however, signalling is rarely instantaneous or direct. Many physically motivated chemotaxis models replace the direct production law $-\Delta v=u$ by an indirect cascade, where the migration signal is not produced directly by the cells but is generated through an intermediate substance, reflecting multi-stage metabolic or kinetic processing \cite{fujie2017application,tao2017critical,wang2026algebraic}. The simplest such cascades introduce a single intermediate variable $w$. We focus on two particular cascades that look formally similar but differ in one decisive physical aspect: whether the equation governing the final migration cue assumes instantaneous equilibration (elliptic) or retains its transient diffusion dynamics (parabolic).

\subsection{Two indirect cascades}
\label{subsec:two-cascades}

Throughout the paper, let
\begin{equation}\label{eq:notation-A}
        A:=I-\Delta,
        \qquad
        A_\tau:=I-\tau\Delta,
        \qquad
        \tau>0,
\end{equation}
be realised on $\R^{N}$, on the flat torus $\T^{N}$, or on a bounded $C^{2}$ domain $\Omega\subset\R^{N}$ with homogeneous Neumann boundary conditions; in the bounded-domain setting all operators below are taken with this boundary condition. The chemotactic sensitivity $\chi>0$ is fixed.

The first model considered is the \emph{parabolic--elliptic--elliptic (PES) cascade},
\begin{equation}\label{eq:PES-intro}\tag{PES}
\begin{cases}
        \partial_t u=\Delta u-\chi\Div(u\nabla c),\\[1mm]
        0=\Delta w-w+u,\\[1mm]
        0=\tau\Delta c-c+w.
\end{cases}
\end{equation}
The second is the \emph{mixed elliptic--parabolic (MEP) cascade},
\begin{equation}\label{eq:MEP-intro}\tag{MEP}
\begin{cases}
        \partial_t u=\Delta u-\chi\Div(u\nabla c),\\[1mm]
        0=\Delta w-w+u,\\[1mm]
        \partial_t c=\Delta c-c+w.
\end{cases}
\end{equation}
In both systems $u$ is the cell density, $w$ is the intermediate substance, and $c$ is the final chemotactic signal sensed by the cells.

In \eqref{eq:PES-intro}, both signal equations are elliptic, meaning the intermediate kinetics equilibrate infinitely fast. The signal variables can thus be eliminated pointwise in time:
\begin{equation}\label{eq:PES-elimination}
        A w=u,\qquad A_\tau c=w
        \quad\Longrightarrow\quad
        c=K_\tau u,\qquad
        K_\tau:=A_\tau^{-1}A^{-1}.
\end{equation}
The operator $K_\tau$ is positive and self-adjoint on $L^{2}$, with Fourier multiplier
\begin{equation}\label{eq:Ktau-multiplier-intro}
        \widehat{K_\tau f}(\xi)
        =
        \frac{\widehat f(\xi)}{(1+|\xi|^{2})(1+\tau|\xi|^{2})}.
\end{equation}
The behaviour at high frequency is
\begin{equation}\label{eq:Ktau-asymptotic-intro}
        \frac{1}{(1+|\xi|^{2})(1+\tau|\xi|^{2})}
        =
        \frac{1}{\tau|\xi|^{4}}
        +
        O(|\xi|^{-6})
        \quad\text{as }|\xi|\to\infty,
\end{equation}
so $K_\tau$ has principal symbol of order $-4$ (we follow the symbol-decay convention; see Convention~\ref{conv:order} below) and its principal homogeneous part is $\tau^{-1}(-\Delta)^{-2}$. The chemotactic drift $\nabla K_\tau$ has symbol of order $-3$.

In \eqref{eq:MEP-intro}, the first elimination still gives $w=A^{-1}u$, but the final signal equation reduces to
\begin{equation}\label{eq:MEP-signal-eq-intro}
        \partial_t c+A c=A^{-1}u.
\end{equation}
Duhamel's formula then yields
\begin{equation}\label{eq:MEP-Duhamel-intro}
        c(t)=e^{-tA}c_0
        +\int_{0}^{t}e^{-(t-s)A}A^{-1}u(s)\,ds.
\end{equation}
Thus, instead of a static fourth-order elliptic map $u(t)\mapsto c(t)$, the MEP elimination produces a Volterra memory operator. The memory part of the drift,
\begin{equation}\label{eq:MEP-Volterra-drift-intro}
        \Vop u(t):=\int_{0}^{t}\nabla e^{-(t-s)A}A^{-1}u(s)\,ds,
\end{equation}
has Fourier multiplier
\begin{equation}\label{eq:MEP-multiplier-intro}
        m(\xi,\theta)
        =\frac{i\xi\,e^{-\theta(1+|\xi|^{2})}}{1+|\xi|^{2}},
        \qquad \theta=t-s>0.
\end{equation}
For every fixed $\theta>0$ the heat factor renders $m(\xi,\theta)$ rapidly decaying in $\xi$. Near the time diagonal,
\begin{equation}\label{eq:MEP-near-diagonal-intro}
        m(\xi,\theta)\longrightarrow\frac{i\xi}{1+|\xi|^{2}}
        \qquad\text{as }\theta\downarrow 0,
\end{equation}
and the limiting multiplier has symbol of order $-1$, the same order as the classical Keller--Segel drift $\nabla(I-\Delta)^{-1}$.

On the other hand, if $u$ is frozen in time then
\begin{equation}\label{eq:MEP-frozen-average-intro}
        \int_{0}^{t}\nabla e^{-\theta A}A^{-1}u\,d\theta
        =\nabla A^{-2}\bigl(I-e^{-tA}\bigr)u,
\end{equation}
which has symbol of order $-3$. Therefore, the MEP drift is, in a precise sense, both Keller--Segel-like near the time diagonal and PES-like in its time average. Reconciling these two competing features is the criticality problem for MEP and lies outside the static elliptic framework that suffices for PES.

\begin{remark}[Convention on Fourier transform and pseudodifferential order]\label{conv:order}
Throughout we use the convention
\[
        \widehat f(\xi)=\int_{\R^{N}}e^{-ix\cdot\xi}f(x)\,dx,
        \qquad
        f(x)=(2\pi)^{-N}\int_{\R^{N}}e^{ix\cdot\xi}\widehat f(\xi)\,d\xi,
\]
so that $\widehat{\Delta f}(\xi)=-|\xi|^{2}\widehat f(\xi)$. For pseudodifferential order we use the symbol-decay convention: a Fourier multiplier whose symbol decays like $|\xi|^{-m}$ at high frequency is said to have \emph{principal order} $-m$. In this convention $K_\tau$ has principal order $-4$ (it is a smoothing operator that gains four derivatives), $\nabla K_\tau$ has principal order $-3$, and the classical Keller--Segel drift $\nabla A^{-1}$ has principal order $-1$. All explicit constants appearing below (in particular the kernel coefficient $1/(8\pi^{2}\tau)$ in dimension four) are computed in this Fourier convention.
\end{remark}

\begin{figure}[htbp]
    \centering
    
    \resizebox{\textwidth}{!}{%
    \begin{tikzpicture}[
        entity/.style = {draw, rectangle, rounded corners, minimum width=2.4cm, minimum height=1.2cm, align=center, fill=blue!5, thick, font=\sffamily},
        process/.style = {->, >=stealth, thick, font=\footnotesize\sffamily, align=center},
        highlight_PES/.style = {draw=red!60, dashed, rounded corners, thick, inner sep=12pt},
        highlight_MEP/.style = {draw=orange!80, dashed, rounded corners, thick, inner sep=12pt},
        labelbox/.style = {font=\sffamily\bfseries, align=left}
    ]

    \def\xU{0}
    \def\xW{4.5}
    \def\xC{9}
    \def\xM{13.5}

    \node[labelbox] (title1) at (\xU-1.2, 1) {(a) Classical Keller--Segel};
    \node[entity] (u1) at (\xU, 0) {Cell density\\$u$};
    \node[entity] (v1) at (\xC, 0) {Signal\\$v$};
    \node[entity] (m1) at (\xM, 0) {Cell movement};

    \draw[process] (u1) -- node[above] {Direct secretion} node[below] {$-\Delta v = u$} (v1);
    \draw[process] (v1) -- node[above] {Chemotaxis} node[below] {$-\chi\nabla\cdot(u\nabla v)$} (m1);

    \node[labelbox] (title2) at (\xU-1.2, -2.5) {(b) Parabolic--Elliptic--Elliptic (PES)};
    \node[entity] (u2) at (\xU, -3.5) {Cell density\\$u$};
    \node[entity] (w2) at (\xW, -3.5) {Intermediate\\$w$};
    \node[entity] (c2) at (\xC, -3.5) {Signal\\$c$};
    \node[entity] (m2) at (\xM, -3.5) {Cell movement};

    \draw[process] (u2) -- node[above] {Fast prod.} node[below] {$-\Delta w+w=u$} (w2);
    \draw[process] (w2) -- node[above] {Fast conv.} node[below] {$-\tau\Delta c+c=w$} (c2);
    \draw[process] (c2) -- node[above] {Chemotaxis} node[below] {$-\chi\nabla\cdot(u\nabla c)$} (m2);

    \begin{scope}[on background layer]
        \node[fit=(w2) (c2), highlight_PES] (pes_box) {};
    \end{scope}
    \node[font=\sffamily\small\bfseries, text=red!80!black, above=3pt of pes_box] {Fast / Elliptic Equilibration};

    \node[labelbox] (title3) at (\xU-1.2, -6.5) {(c) Mixed Elliptic--Parabolic (MEP)};
    \node[entity] (u3) at (\xU, -7.5) {Cell density\\$u$};
    \node[entity] (w3) at (\xW, -7.5) {Intermediate\\$w$};
    \node[entity, fill=orange!15, draw=orange!80!black] (c3) at (\xC, -7.5) {Signal\\$c$};
    \node[entity] (m3) at (\xM, -7.5) {Cell movement};

    \draw[process] (u3) -- node[above] {Fast prod.} node[below] {$-\Delta w+w=u$} (w3);
    
    \draw[process, draw=orange!80!black, line width=1.6pt] (w3) -- node[above] {Time evolution} node[below] {$\partial_t c = \Delta c - c + w$} (c3);
    
    \draw[process] (c3) -- node[above] {Chemotaxis} node[below] {$-\chi\nabla\cdot(u\nabla c)$} (m3);

    \begin{scope}[on background layer]
        \node[fit=(w3) (c3), highlight_MEP] (mep_box) {};
    \end{scope}
    \node[font=\sffamily\small\bfseries, text=orange!80!black, above=3pt of mep_box] {Transient Memory (Volterra / Parabolic)};

    \end{tikzpicture}%
    }
    
    \caption{Conceptual schematic of the structural dichotomy in signal-generation mechanisms. \textbf{(a)} The classical Keller--Segel system assumes instantaneous, direct secretion of the chemoattractant $v$. \textbf{(b)} The PES cascade incorporates an intermediate metabolic step $w$, but assumes both intermediate and final signals equilibrate infinitely fast (elliptic), mathematically reducing to a static fourth-order smoothing interaction. \textbf{(c)} The MEP cascade retains the transient time-evolution (parabolic) of the final signal $c$. This breaks the static elliptic reduction, introducing a genuine Volterra memory effect that forces a mixed space--time threshold problem.}
    \label{fig:mechanisms_schematic}
\end{figure}
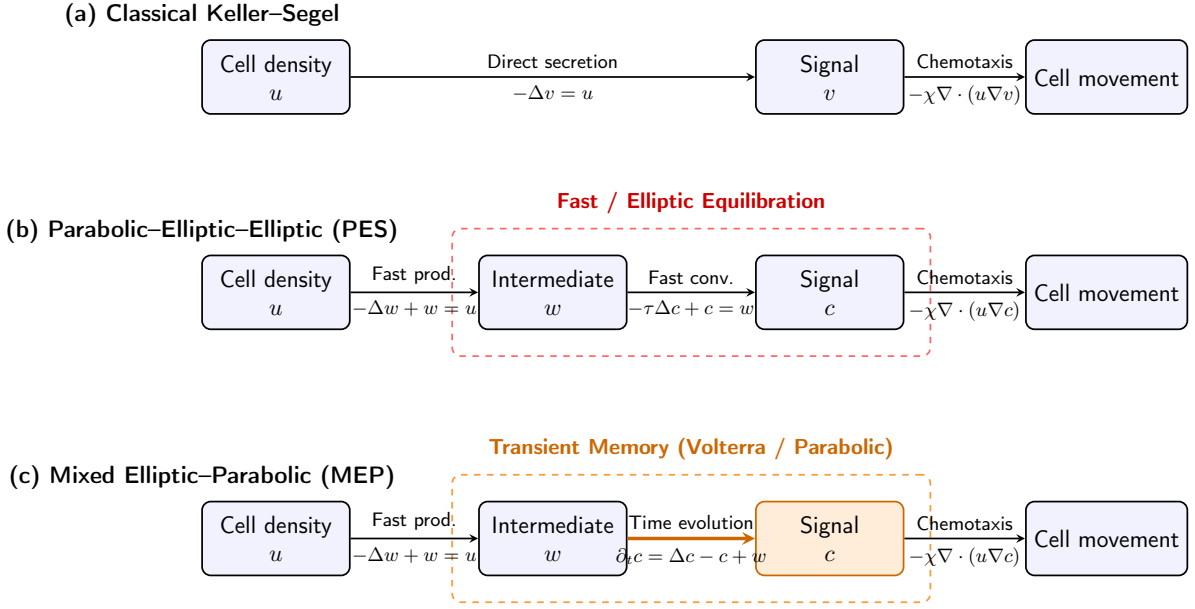

\subsection{Main structural theorem}
\label{subsec:main-thm}

The central result of the paper makes the PES/MEP dichotomy precise.

\begin{theorem}[Structural dichotomy]\label{thm:main}
Let $N\ge 1$, $\tau>0$, $\chi>0$, and let $A=I-\Delta$, $A_\tau=I-\tau\Delta$ be realised on $\R^{N}$, on $\T^{N}$, or on a bounded $C^{2}$ domain with homogeneous Neumann boundary conditions. Then:
\begin{enumerate}[label=(\roman*),leftmargin=2em]
\item \textup{(PES structure.)} The PES cascade \eqref{eq:PES-intro} reduces to $c=K_\tau u$ with $K_\tau=A_\tau^{-1}A^{-1}$. The operator $K_\tau$ is positive and self-adjoint on $L^{2}$. It is a smoothing operator of principal order $-4$ whose principal homogeneous part is $\tau^{-1}(-\Delta)^{-2}$.
\item \textup{(PES scaling.)} The high-frequency parabolic scaling of PES is $u_\lambda(x,t)=\lambda^{4}u(\lambda x,\lambda^{2}t)$, yielding the scaling-critical Lebesgue exponent $q_c=N/4$ and the mass-critical dimension $N=4$.
\item \textup{(PES log-kernel in $N=4$.)} On $\R^{4}$, the integral kernel $G_\tau$ of $K_\tau$ admits the local expansion
\[
        G_\tau(x)=\frac{1}{8\pi^{2}\tau}\log\frac{1}{|x|}+R_\tau(x),
\]
where $R_\tau\in C^{1}(\R^{4}\setminus\{0\})\cap L^{\infty}_{\loc}(\R^{4})$. The same expansion holds in the interior of any bounded $C^{\infty}$ Neumann domain and on $\T^{4}$ \emph{(}with $C^{2}$ regularity sufficient at all points away from the boundary, contingent on the corresponding interior Green-kernel regularity; see Proposition~\ref{prop:log-kernel-R4}\emph{)}.
\item \textup{(MEP structure.)} The MEP cascade \eqref{eq:MEP-intro} eliminates to the Volterra representation \eqref{eq:MEP-Duhamel-intro}. The memory drift multiplier \eqref{eq:MEP-multiplier-intro} satisfies the near-diagonal limit \eqref{eq:MEP-near-diagonal-intro}; the limit has principal order $-1$.
\item \textup{(Class separation.)} PES is a static fourth-order elliptic chemotaxis system. MEP is a Volterra-memory chemotaxis system; it cannot be assigned PES criticality by a static fourth-order elliptic scaling argument.
\end{enumerate}
\end{theorem}

The constituent parts of Theorem~\ref{thm:main} are proved in Sections~\ref{sec:PES}--\ref{sec:MEP}. Item (i) is Lemma~\ref{lem:Ktau-smoothing} and Proposition~\ref{prop:Ktau-principal-symbol}; item (ii) is Proposition~\ref{prop:PES-scaling}; item (iii) is Proposition~\ref{prop:log-kernel-R4}; item (iv) is Propositions~\ref{prop:MEP-Volterra} and~\ref{prop:MEP-near-diagonal}; item (v) follows from items (i)--(iv).

\subsection{Free energy and concentration-scaling candidate mass}
\label{subsec:free-energy-intro}

Because $K_\tau$ is self-adjoint, PES admits the gradient-flow free energy
\begin{equation}\label{eq:PES-free-energy-intro}
        \Free_{\mathrm{PES}}[u]
        :=\int_{\Omega}u\log u\,dx
        -\frac{\chi}{2}\int_{\Omega}u\,K_\tau u\,dx.
\end{equation}
Smooth positive solutions of the eliminated PES equation satisfy the entropy-dissipation identity
\begin{equation}\label{eq:PES-dissipation-intro}
        \frac{d}{dt}\Free_{\mathrm{PES}}[u(t)]
        =-\int_{\Omega}u\,\bigl|\nabla(\log u-\chi K_\tau u)\bigr|^{2}\,dx
        \le 0;
\end{equation}
see Proposition~\ref{prop:PES-dissipation}. In dimension $N=4$, the logarithmic local singularity of $K_\tau$ produces the same kind of variational concentration scenario that, in the two-dimensional classical Keller--Segel problem, leads to the critical mass $8\pi/\chi$.

For a one-point concentrating sequence
\begin{equation}\label{eq:concentration-intro}
        u_\varepsilon(x)=\varepsilon^{-4}U\!\left(\tfrac{x-x_0}{\varepsilon}\right),
        \qquad
        U\in C_c^{\infty}(B_{1}),
        \ U\ge 0,
        \ \int_{\R^{4}}U\,dx=M,
\end{equation}
located at an interior point $x_0$, an explicit computation (Proposition~\ref{prop:candidate-mass}) gives
\begin{equation}\label{eq:concentration-energy-intro}
        \Free_{\mathrm{PES}}[u_\varepsilon]
        =\left(4M-\frac{\chi M^{2}}{16\pi^{2}\tau}\right)\log\frac{1}{\varepsilon}+O(1).
\end{equation}
The leading coefficient vanishes at
\begin{equation}\label{eq:M-star-intro}
        M_{\!*}=\frac{64\pi^{2}\tau}{\chi}
\end{equation}
which we identify as the \emph{concentration-scaling candidate} for the four-dimensional PES critical mass. We emphasise that the computation \eqref{eq:concentration-energy-intro}--\eqref{eq:M-star-intro} is a scaling identification, not a sharp threshold theorem. To promote $M_{\!*}$ to a sharp threshold requires the corresponding sharp $K_\tau$-adapted logarithmic Hardy--Littlewood--Sobolev / Adams inequality; we formulate this as Problem~\ref{prob:sharp-logHLS} in Section~\ref{sec:open}.

\subsection{Position relative to the literature}
\label{subsec:position}

The classical two-dimensional Keller--Segel system, initiated by the foundational works \cite{keller1970initiation,keller1971model} and comprehensively surveyed in \cite{horstmann20031970,hillen2009user,bellomo2015toward,wang2025analysis,winkler2016two}, has been intensively studied since the 1990s. Its sharp critical mass $8\pi/\chi$ was established for the parabolic--elliptic case in \cite{blanchet2006two} via the logarithmic-HLS inequality of \cite{beckner1993sharp}; the classification of blow-up profiles and infinite-time aggregation has been further developed through the two-dimensional critical-mass theory in \cite{blanchet2008infinite,liang2025global,raphael2014stability,nagai1997application}.

Indirect signal-production cascades have received increasing attention. Tao and Winkler \cite{tao2017critical} studied a related indirect-production mechanism in two dimensions, in which the migration signal is generated through an intermediate component with temporal relaxation. Their model has the form, up to normalisation and notation,
\[
        u_t=\Delta u-\Div(u\nabla v),
        \qquad
        0=\Delta v-\mu(t)+w,
        \qquad
        \tau w_t+\delta w=u,
\]
on the unit disk with homogeneous Neumann boundary conditions. They proved a critical-mass phenomenon for \emph{infinite-time} aggregation, in sharp contrast with the finite-time blow-up mechanism of the classical parabolic--elliptic Keller--Segel system. This indirect mechanism has subsequently inspired investigations along several complementary directions: models featuring memory effects akin to MEP have been studied on the whole plane \cite{xiang2025critical}, bounded domains \cite{yang2026critical,laurenccot2024singular,wang2025analysis1}, and in higher-dimensional blow-up settings \cite{mao2025critical,mao2025finite,mao2025finite1,wang2026damage}. Furthermore, a sharp criterion for global existence versus finite-time blow-up in supercritical dimensions was recently established \cite{soga2026sharp}, alongside the identification of a general family of mass-critical systems \cite{winkler2022family,yu2026pattern} and observations that haptotaxis minimally impacts global dynamics in such models \cite{chen2024negligibility}.

The present MEP system is not identical to that early model: here the final migration signal satisfies a spatially parabolic equation $\partial_t c=\Delta c-c+w$, and the mathematical analysis of such MEP systems frequently relies on the semigroup and maximal-regularity viewpoint for parabolic equations \cite{pazy2012semigroups,henry2006geometric,lunardi2012analytic,wang2026breakdown,amann1995linear,pruss2016moving}. Moreover, the Volterra memory structure inherent in MEP is linked to spatial movement models with distributed memory \cite{shi2021spatial} and facilitates the rigorous convergence from indirect to direct signal production via singular limits \cite{bui2026parabolic,wang2025multi,liu2025global}. Thus, Tao--Winkler's work should be viewed as a guiding example showing that indirect production can fundamentally alter critical dynamics, rather than as a direct instance of the MEP cascade analysed below.

Related indirect-production and multi-chemical chemotaxis systems have been studied in several directions; see, for example, \cite{fujie2017application} for the use of Adams-type inequalities in two-chemical chemotaxis systems, and \cite{nagai2000chemotactic,gao2022rolling,mao2024dirac} for the analysis of chemotactic collapse and Dirac-type aggregation with full mass. More generally, logarithmic mechanisms in related settings are closely tied to the Adams inequality and higher-order potential theory \cite{adams1988sharp,adams2003sobolev,hebey2000nonlinear,yu2026beyond,lin1998classification,martinazzi2009classification}, together with logarithmic Hardy--Littlewood--Sobolev ideas \cite{carlen1992competing,lieb2001analysis}. Notably, sharp Adams-type inequalities on $\mathbb{R}^n$ \cite{ruf2013sharp} are essential for underpinning the four-dimensional PES theory. We also mention \cite{laurenccot2017finite} as a representative finite-time blow-up result for parabolic--parabolic Keller--Segel systems with critical diffusion. Additionally, weak solution frameworks for parabolic--elliptic chemotaxis systems have been developed to provide functional settings for continuing solutions beyond blow-up \cite{senba2002weak}. By contrast, while the four-dimensional critical regime relevant to PES has seen recent analytical progress \cite{fujie2017application,yu2026rigorous,fujie2019blowup,hosono2025global}, fourth-order elliptic chemotaxis cascades---i.e.\ the PES system---have overall received comparatively little structural analysis, despite arising naturally in models of biological signalling with two equilibrating intermediate steps.

The present paper does not reprove or generalise the Tao--Winkler critical-mass theorem, and does not assert that fourth-order or logarithmic critical-mass phenomena in chemotaxis are entirely unexplored: \cite{fujie2017application} already exploits Adams-type inequalities for two-chemical systems in dimension four. The novelty of our contribution is structural rather than dynamical: we isolate the operator-theoretic distinction between the fully elliptic cascade and the mixed elliptic--parabolic cascade, and identify the exact static fourth-order interaction $K_\tau=A_\tau^{-1}A^{-1}$ associated with the PES model, together with the corresponding logarithmic kernel coefficient $1/(8\pi^{2}\tau)$ in dimension four and the resulting concentration-scaling candidate mass $M_{\!*}=64\pi^{2}\tau/\chi$. By contrast, the mixed elliptic--parabolic cascade is shown to belong to a Volterra-memory class whose threshold problem is genuinely mixed in space and time, and whose criticality cannot be inferred from a static elliptic scaling argument.

\subsection{Organisation}

We prove the operator-theoretic structure, the high-frequency scaling, the four-dimensional kernel expansion, the entropy-dissipation identity, the concentration-scaling identification of $M_{\!*}$, and the Volterra space--time estimates for MEP, together with conditional sharp lower bounds and continuation criteria. Figure~\ref{fig:mechanisms_schematic} provides a schematic of the structural dichotomy in signal-generation mechanisms.

Section~\ref{sec:PES} develops the operator-theoretic structure of PES, including spectral and Fourier representation of $K_\tau$, scaling, and the four-dimensional logarithmic kernel expansion. Section~\ref{sec:PES-energy} establishes the free-energy structure, entropy dissipation, and concentration-scaling identification of $M_{\!*}$. Section~\ref{sec:MEP} analyses MEP via Duhamel formula, multiplier asymptotics, and space--time Volterra estimates. Section~\ref{sec:open} formulates the open threshold problems and records the resulting conditional statements.

\section{Operator structure of PES}
\label{sec:PES}

We begin the analysis of \eqref{eq:PES-intro} by establishing the operator-theoretic properties of the eliminated interaction $K_\tau=A_\tau^{-1}A^{-1}$ in \eqref{eq:PES-elimination}.

\subsection{Spectral and Fourier representation}
\label{subsec:Ktau-rep}

On $\R^{N}$, $K_\tau$ is the Fourier multiplier
\begin{equation}\label{eq:Ktau-symbol}
        \widehat{K_\tau f}(\xi)=m_\tau(\xi)\widehat f(\xi),
        \qquad
        m_\tau(\xi):=\frac{1}{(1+|\xi|^{2})(1+\tau|\xi|^{2})}.
\end{equation}
Thus $0<m_\tau(\xi)\le 1$ and $m_\tau(\xi)\sim (\tau|\xi|^{4})^{-1}$ as $|\xi|\to\infty$.

On $\T^{N}$ with $f(x)=\sum_{k\in\Z^{N}}\widehat f_k e^{ik\cdot x}$,
\begin{equation}\label{eq:Ktau-torus}
        K_\tau f
        =\sum_{k\in\Z^{N}}
        \frac{\widehat f_k}{(1+|k|^{2})(1+\tau|k|^{2})}\,e^{ik\cdot x}.
\end{equation}
On a bounded $C^{2}$ Neumann domain $\Omega\subset\R^{N}$, let $\{\varphi_j\}_{j\ge 0}$ be the $L^{2}$-orthonormal basis of Neumann eigenfunctions of $-\Delta$ with eigenvalues $0=\lambda_0<\lambda_1\le\lambda_2\le\cdots$. Then
\begin{equation}\label{eq:Ktau-spectral}
        K_\tau f
        =\sum_{j=0}^{\infty}
        \mu_j^{(\tau)}\langle f,\varphi_j\rangle_{L^{2}}\varphi_j,
        \qquad
        \mu_j^{(\tau)}:=\frac{1}{(1+\lambda_j)(1+\tau\lambda_j)}.
\end{equation}

For $s\in\R$, we denote by $\Hil_A^{s}$ the spectral Sobolev space associated with the self-adjoint operator $A=I-\Delta$ on the relevant manifold; thus
\[
        \|f\|_{\Hil_A^{s}}^{2}
        :=\langle A^{s}f,f\rangle
        =\sum_{j}(1+\lambda_j)^{s}|\langle f,\varphi_j\rangle|^{2}
\]
on $\Omega$ (with the obvious modifications on $\R^{N}$ or $\T^{N}$), which coincides with the usual Sobolev space $H^{s}$ in the standard range; cf.~\cite{triebel1992theory}.

\begin{lemma}[Positivity, self-adjointness, and fourth-order smoothing]\label{lem:Ktau-smoothing}
Let $\tau>0$. The operator $K_\tau=A_\tau^{-1}A^{-1}$ is positive and self-adjoint on $L^{2}$. Moreover, for every $s\in\R$,
\begin{equation}\label{eq:Ktau-Hs-smoothing}
        \|K_\tau f\|_{\Hil_A^{s+4}}\le \max\{1,\tau^{-1}\}\,\|f\|_{\Hil_A^{s}}.
\end{equation}
On $\R^{N}$ this reduces to the standard Sobolev estimate $\|K_\tau f\|_{H^{s+4}}\le \max\{1,\tau^{-1}\}\|f\|_{H^{s}}$.
\end{lemma}

\begin{proof}
Positivity and self-adjointness follow from the spectral representation \eqref{eq:Ktau-spectral} (resp.\ from $0<m_\tau(\xi)\le 1$ on $\R^{N}$ or $\T^{N}$). For the smoothing estimate, observe
\[
        (1+\lambda_j)^{2}\mu_j^{(\tau)}
        =\frac{1+\lambda_j}{1+\tau\lambda_j}
        \le \max\{1,\tau^{-1}\}\qquad\text{for every }j\ge 0,
\]
since $(1+\lambda)/(1+\tau\lambda)$ is a Möbius function bounded between $1$ and $\tau^{-1}$ on $[0,\infty)$. Hence
\[
        \|K_\tau f\|_{\Hil_A^{s+4}}^{2}
        =\sum_{j}(1+\lambda_j)^{s+4}\bigl(\mu_j^{(\tau)}\bigr)^{2}|\langle f,\varphi_j\rangle|^{2}
        \le \max\{1,\tau^{-2}\}\,\|f\|_{\Hil_A^{s}}^{2}.
\]
The Euclidean case follows identically from Plancherel and the multiplier bound $(1+|\xi|^{2})^{2}m_\tau(\xi)\le\max\{1,\tau^{-1}\}$.
\end{proof}

The associated quadratic form is non-negative: on $\R^{N}$,
\begin{equation}\label{eq:Ktau-quadratic-Fourier}
        \int_{\R^{N}}f\,K_\tau f\,dx
        =\int_{\R^{N}}\frac{|\widehat f(\xi)|^{2}}{(1+|\xi|^{2})(1+\tau|\xi|^{2})}\,d\xi\ge 0,
\end{equation}
and on $\Omega$,
\begin{equation}\label{eq:Ktau-quadratic-spectral}
        \int_{\Omega}f\,K_\tau f\,dx
        =\sum_{j=0}^{\infty}\mu_j^{(\tau)}|\langle f,\varphi_j\rangle|^{2}\ge 0.
\end{equation}

\subsection{Principal symbol and drift order}
\label{subsec:principal-symbol}

The high-frequency expansion of the symbol is
\begin{equation}\label{eq:m-tau-expansion}
        m_\tau(\xi)
        =\frac{1}{\tau|\xi|^{4}}\left[1-\frac{1+\tau^{-1}}{|\xi|^{2}}+O(|\xi|^{-4})\right]
        \quad\text{as }|\xi|\to\infty.
\end{equation}
Equivalently, microlocally at high frequency,
\begin{equation}\label{eq:Ktau-principal-decomp}
        K_\tau=\tau^{-1}(-\Delta)^{-2}+\Op(S^{-6}),
\end{equation}
where the remainder is a pseudodifferential operator whose symbol lies in the standard Hörmander class $S^{-6}$ of symbols decaying at order $-6$.

\begin{proposition}[Principal symbol of PES interaction]\label{prop:Ktau-principal-symbol}
The eliminated PES interaction $K_\tau$ is a fourth-order elliptic smoothing operator. Its principal symbol is $\tau^{-1}|\xi|^{-4}$, and its principal homogeneous part is $\tau^{-1}(-\Delta)^{-2}$. Consequently, the chemotactic drift $\nabla K_\tau$ has principal symbol $i\xi/(\tau|\xi|^{4})$ and principal order $-3$.
\end{proposition}

\begin{proof}
The asymptotic expansion \eqref{eq:m-tau-expansion} is obtained by writing
$
        m_\tau(\xi)=(\tau|\xi|^{4})^{-1}(1+|\xi|^{-2})^{-1}(1+(\tau|\xi|^{2})^{-1})^{-1}
$
and expanding each factor in a geometric series for $|\xi|^{2}>\max\{1,\tau^{-1}\}$. The corresponding pseudodifferential decomposition \eqref{eq:Ktau-principal-decomp} is then standard. Differentiation in $x$ multiplies the symbol by $i\xi$, lowering the symbol-decay order from $4$ to $3$.
\end{proof}

\subsection{Scaling and the critical Lebesgue exponent}
\label{subsec:PES-scaling}

The full Bessel-product operator $K_\tau$ is not exactly scale-invariant because $A$ and $A_\tau$ contain zeroth-order terms; however, its principal homogeneous part $\tau^{-1}(-\Delta)^{-2}$ is exactly invariant under dilations. The following scaling discussion concerns this principal homogeneous regime, which dictates the high-frequency dynamics.

For $\lambda>0$, define
\begin{equation}\label{eq:PES-scaling-def}
        u_\lambda(x,t):=\lambda^{\alpha}u(\lambda x,\lambda^{2}t).
\end{equation}
Then $\partial_t u_\lambda=\lambda^{\alpha+2}(\partial_t u)(\lambda x,\lambda^{2}t)$ and $\Delta u_\lambda=\lambda^{\alpha+2}(\Delta u)(\lambda x,\lambda^{2}t)$. Since the principal part of $K_\tau$ has symbol of order $-4$, it transforms under dilations as
\begin{equation}\label{eq:Ktau-scaling-principal}
        (K_\tau u_\lambda)(x,t)
        \sim\lambda^{\alpha-4}(K_\tau u)(\lambda x,\lambda^{2}t),
\end{equation}
in the sense that the equality holds exactly with $K_\tau$ replaced by its principal part $\tau^{-1}(-\Delta)^{-2}$, with lower-order terms suppressed by an additional factor $\lambda^{-2}$ at large $\lambda$. Consequently $\nabla K_\tau u_\lambda\sim\lambda^{\alpha-3}(\nabla K_\tau u)(\lambda x,\lambda^{2}t)$ and the nonlinear aggregation term scales as
\begin{equation}\label{eq:nonlinear-scaling}
        \Div(u_\lambda\nabla K_\tau u_\lambda)
        \sim\lambda^{2\alpha-2}\bigl[\Div(u\nabla K_\tau u)\bigr](\lambda x,\lambda^{2}t).
\end{equation}
Balancing with $\partial_t u_\lambda$ and $\Delta u_\lambda$ requires $\alpha+2=2\alpha-2$, i.e.\ $\alpha=4$.

\begin{proposition}[Scaling of PES]\label{prop:PES-scaling}
The principal parabolic scaling of PES is
\begin{equation}\label{eq:PES-scaling-final}
        u_\lambda(x,t)=\lambda^{4}u(\lambda x,\lambda^{2}t).
\end{equation}
Under this scaling, for $1\le q<\infty$,
\begin{equation}\label{eq:Lq-scaling-final}
        \|u_\lambda(\cdot,t)\|_{L^{q}(\R^{N})}
        =\lambda^{4-N/q}\,\|u(\cdot,\lambda^{2}t)\|_{L^{q}(\R^{N})}.
\end{equation}
Therefore the scaling-critical Lebesgue exponent is $q_c=N/4$, and the mass-critical dimension is $N=4$.
\end{proposition}

\begin{proof}
The scaling exponent $\alpha=4$ was derived above. For \eqref{eq:Lq-scaling-final},
\[
        \|u_\lambda(\cdot,t)\|_{L^{q}}^{q}
        =\int_{\R^{N}}\lambda^{4q}|u(\lambda x,\lambda^{2}t)|^{q}\,dx
        =\lambda^{4q-N}\int_{\R^{N}}|u(y,\lambda^{2}t)|^{q}\,dy.
\]
Taking the $q$-th root yields \eqref{eq:Lq-scaling-final}. The mass case $q=1$ gives invariance precisely at $N=4$.
\end{proof}

\begin{remark}[Comparison with classical Keller--Segel]\label{rem:KS-comparison}
For the classical parabolic--elliptic Keller--Segel system, the drift $\nabla(-\Delta)^{-1}$ has principal order $-1$ and the scaling balance gives $\alpha=2$, so the critical Lebesgue exponent is $q_c=N/2$, with mass criticality in $N=2$. The PES drift $\nabla K_\tau$ has principal order $-3$, the scaling exponent is $\alpha=4$, and the critical exponent is $q_c=N/4$. In particular, $L^{N/2}$ is \emph{not} the scaling-critical space for PES.
\end{remark}

\subsection{Logarithmic kernel in dimension four}
\label{subsec:log-kernel}

We now identify the local singularity of the integral kernel of $K_\tau$ in dimension four. The argument uses partial fractions together with the classical asymptotic expansion of the four-dimensional Bessel potential of order two.

\begin{lemma}[Bessel kernel $G_2$ in $\R^{4}$]\label{lem:bessel-G2}
Let $N=4$. The kernel
\begin{equation}\label{eq:G2-definition}
        G_2(x):=\mathcal{F}^{-1}\!\left[\frac{1}{1+|\xi|^{2}}\right](x),\qquad x\in\R^{4},
\end{equation}
is a positive radial function in $C^{\infty}(\R^{4}\setminus\{0\})$ which decays exponentially at infinity and admits the local expansion
\begin{equation}\label{eq:G2-expansion}
        G_2(x)=\frac{1}{4\pi^{2}|x|^{2}}-\frac{1}{8\pi^{2}}\log\frac{1}{|x|}+\rho_2(x),
\end{equation}
where $\rho_2$ is continuous on $\R^{4}$ and bounded in a neighbourhood of the origin.
\end{lemma}

\begin{proof}
The smoothness of $G_2$ away from $0$, positivity, and exponential decay are classical \cite{stein1970singular}. The function $G_2$ solves
\begin{equation}\label{eq:G2-equation}
        -\Delta G_2+G_2=\delta_0\qquad\text{in }\mathcal D'(\R^{4}).
\end{equation}
Let $G_0(x):=1/(4\pi^{2}|x|^{2})$ be the Newtonian fundamental solution of $-\Delta$ on $\R^{4}$, so $-\Delta G_0=\delta_0$. Then $R:=G_2-G_0$ satisfies
\begin{equation}\label{eq:R-equation}
        -\Delta R=-G_2\qquad\text{in }\R^{4}\setminus\{0\},
\end{equation}
with $R(x)=O(\log(1/|x|))$ as $|x|\to 0$ by the heat-kernel representation of $G_2$ \cite{stein1970singular}. A direct radial computation gives $\Delta(\log|x|)=2/|x|^{2}$ in $\R^{4}$. Hence the function $w(x):=-(8\pi^{2})^{-1}\log(1/|x|)$ satisfies
\[
    -\Delta w=-\frac{1}{4\pi^{2}|x|^{2}}=-G_0(x),
\]
and consequently
\[
        -\Delta(R-w)=-G_2+G_0=-R.
\]
The right-hand side has at most logarithmic singularity, so standard elliptic regularity gives $R-w\in L^{\infty}_{\loc}(\R^{4})$. Setting $\rho_2:=R-w\in C^{0}(\R^{4})$ yields \eqref{eq:G2-expansion}.
\end{proof}

\begin{figure}[!h]
    \centering
    \includegraphics[width=0.8\textwidth]{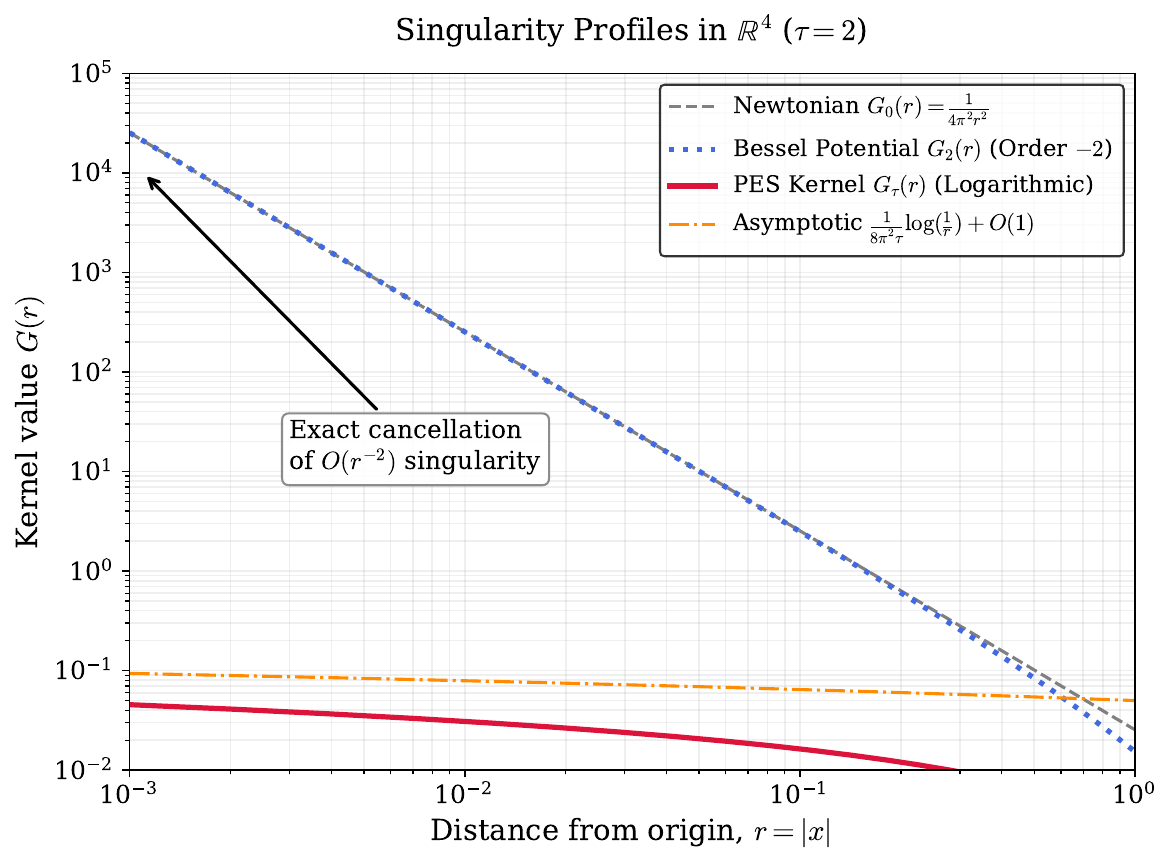}
    \caption{Comparison of radial singularity profiles in $\mathbb{R}^4$ ($\tau=2$). The classical Newtonian potential and the single Bessel potential $G_2(r)$ exhibit a steep $O(r^{-2})$ singularity at the origin. In contrast, the eliminated PES interaction kernel $G_\tau(r)$ demonstrates an exact algebraic cancellation of this leading singularity, resulting in a softened, purely logarithmic divergence. This structural softening is the geometric mechanism that shifts the mass-critical dimension of the cascade from $N=2$ to $N=4$.}
    \label{fig:kernel_singularity}
\end{figure}

\begin{proposition}[Logarithmic kernel of $K_\tau$ in $N=4$]\label{prop:log-kernel-R4}
Let $N=4$ and $\tau>0$. The integral kernel $G_\tau$ of $K_\tau$ on $\R^{4}$ admits the local expansion
\begin{equation}\label{eq:Gtau-expansion-R4}
        G_\tau(x)=\frac{1}{8\pi^{2}\tau}\log\frac{1}{|x|}+R_\tau(x),
\end{equation}
where $R_\tau\in C^{0}(\R^{4})$ is bounded in a neighbourhood of the origin. The same expansion holds on $\T^{4}$ for $x$ in any fundamental domain. On a bounded $C^{\infty}$ Neumann domain $\Omega\subset\R^{4}$, the integral kernel $G_\tau(x,y)$ of the Neumann realisation of $K_\tau$ admits, for every compact set $\mathcal K\Subset\Omega$,
\begin{equation}\label{eq:Gtau-expansion-domain}
        G_\tau(x,y)=\frac{1}{8\pi^{2}\tau}\log\frac{1}{|x-y|}+R_\tau(x,y),\qquad (x,y)\in\mathcal K\times\mathcal K,
\end{equation}
with $R_\tau\in L^{\infty}(\mathcal K\times\mathcal K)$. The same conclusion holds on bounded $C^{2}$ Neumann domains under the corresponding interior Green-kernel regularity assumption \emph{(}see Remark~\ref{rem:boundary-low-regularity}\emph{)}.
\end{proposition}

\begin{proof}
\emph{Step 1: Partial-fraction decomposition for $\tau\ne 1$.}
For $\tau\ne 1$, the partial-fraction identity is
\begin{equation}\label{eq:partial-fractions}
        \frac{1}{(1+|\xi|^{2})(1+\tau|\xi|^{2})}
        =\frac{1}{1-\tau}\cdot\frac{1}{1+|\xi|^{2}}
        -\frac{\tau}{1-\tau}\cdot\frac{1}{1+\tau|\xi|^{2}}.
\end{equation}
The inverse Fourier transform of $(1+\tau|\xi|^{2})^{-1}$ is obtained by the dilation $\xi=\eta/\sqrt{\tau}$:
\[
        \mathcal{F}^{-1}\!\left[\frac{1}{1+\tau|\xi|^{2}}\right]\!(x)
        =\tau^{-2}G_2(x/\sqrt{\tau})\qquad\text{in }\R^{4}.
\]
Therefore
\begin{equation}\label{eq:Gtau-partial-fraction}
        G_\tau(x)=\frac{1}{1-\tau}\!\left[G_2(x)-\frac{1}{\tau}\,G_2(x/\sqrt{\tau})\right].
\end{equation}

\emph{Step 2: Cancellation of the leading $|x|^{-2}$ singularity.}
By Lemma~\ref{lem:bessel-G2},
\[
        G_2(x)=\frac{1}{4\pi^{2}|x|^{2}}-\frac{1}{8\pi^{2}}\log\frac{1}{|x|}+\rho_2(x),
\]
and the same lemma applied at $x/\sqrt{\tau}$ gives
\[
        G_2(x/\sqrt{\tau})
        =\frac{\tau}{4\pi^{2}|x|^{2}}-\frac{1}{8\pi^{2}}\log\frac{\sqrt{\tau}}{|x|}+\rho_2(x/\sqrt{\tau}).
\]
Hence
\[
        \frac{1}{\tau}\,G_2(x/\sqrt{\tau})
        =\frac{1}{4\pi^{2}|x|^{2}}
        -\frac{1}{8\pi^{2}\tau}\log\frac{1}{|x|}
        -\frac{\log\sqrt{\tau}}{8\pi^{2}\tau}
        +\tau^{-1}\rho_2(x/\sqrt{\tau}),
\]
and therefore
\[
        G_2(x)-\frac{1}{\tau}\,G_2(x/\sqrt{\tau})
        =\Bigl(\frac{1}{8\pi^{2}\tau}-\frac{1}{8\pi^{2}}\Bigr)\log\frac{1}{|x|}+\tilde\rho(x)
        =\frac{1-\tau}{8\pi^{2}\tau}\log\frac{1}{|x|}+\tilde\rho(x),
\]
with $\tilde\rho$ continuous and bounded near the origin. Combining with \eqref{eq:Gtau-partial-fraction},
\[
        G_\tau(x)=\frac{1}{8\pi^{2}\tau}\log\frac{1}{|x|}+R_\tau(x),
        \qquad R_\tau(x):=\frac{\tilde\rho(x)}{1-\tau},
\]
which proves \eqref{eq:Gtau-expansion-R4} for $\tau\ne 1$. The crucial cancellation is that the $|x|^{-2}$ coefficient in $\frac{1}{\tau}G_2(x/\sqrt{\tau})$ matches that of $G_2(x)$ identically, eliminating the leading singularity. (see Figure~\ref{fig:kernel_singularity})

\emph{Step 3: The case $\tau=1$.}
If $\tau=1$, then $K_1=A^{-2}$ has Fourier multiplier $(1+|\xi|^{2})^{-2}$, and its kernel is the Bessel potential of order four on $\R^{4}$. By the heat-kernel representation \cite{stein1970singular},
\[
        G_1(x)
        =\mathcal{F}^{-1}\!\left[\frac{1}{(1+|\xi|^{2})^{2}}\right]\!(x)
        =\frac{1}{(4\pi)^{2}\Gamma(2)}\int_0^{\infty}t^{-1}e^{-t-|x|^{2}/(4t)}\,dt
        =\frac{1}{8\pi^{2}}K_0(|x|),
\]
where $K_0$ is the modified Bessel function of the second kind of order zero. The classical asymptotics
\[
        K_0(r)=\log\frac{1}{r}+(\log 2-\gamma_E)+O(r^{2}\log r)\qquad\text{as }r\downarrow 0
\]
(with $\gamma_E$ Euler's constant) give
\[
        G_1(x)=\frac{1}{8\pi^{2}}\log\frac{1}{|x|}+R_1(x),
\]
with $R_1$ continuous and bounded near $0$, in agreement with \eqref{eq:Gtau-expansion-R4} at $\tau=1$. Continuity of $R_\tau$ in $\tau$ follows from continuity in $\tau$ of the multiplier $m_\tau$ on every set $\{|\xi|\ge\varepsilon\}$.

\emph{Step 4: Torus.}
On $\T^{4}=\R^{4}/(2\pi\Z)^{4}$, the kernel is
\[
        G_\tau^{\T}(x)=\sum_{k\in(2\pi\Z)^{4}}G_\tau(x+k),
\]
the sum over $k\ne 0$ defining a $C^{\infty}$ function on any fundamental domain. Hence \eqref{eq:Gtau-expansion-R4} holds on $\T^{4}$ with the same leading coefficient.

\emph{Step 5: Bounded $C^{\infty}$ Neumann domain.}
For a bounded $C^{\infty}$ Neumann domain $\Omega\subset\R^{4}$, the standard interior parametrix construction for elliptic pseudodifferential operators (cf.~\cite{taylor2010pseudodifferential}) decomposes the Neumann realisation of $K_\tau$ on $\Omega$ as a Euclidean parametrix with kernel $G_\tau$ (from $\R^{4}$) plus a smoothing remainder of order $-\infty$ on any interior compact $\mathcal K$. Inserting \eqref{eq:Gtau-expansion-R4} into the parametrix gives \eqref{eq:Gtau-expansion-domain}.
\end{proof}

\begin{remark}[Low-regularity domains]\label{rem:boundary-low-regularity}
The interior parametrix in Step~5 uses smooth-coefficient pseudodifferential calculus, which applies cleanly on $C^{\infty}$ domains. For bounded $C^{2}$ Neumann domains, the analogue of \eqref{eq:Gtau-expansion-domain} should still hold with the same leading coefficient at any interior compact $\mathcal K\Subset\Omega$, by a low-regularity Green-kernel argument for the operator $A_\tau A=(I-\tau\Delta)(I-\Delta)$; we do not record the details. Concentration computations later in the paper are stated for interior concentration points $x_0\in\Omega$ with $\dist(x_0,\partial\Omega)>0$, so only the interior expansion is used.
\end{remark}

\begin{remark}[Boundary effects on concentration]\label{rem:boundary}
The expansion \eqref{eq:Gtau-expansion-domain} is an \emph{interior} expansion. Concentration near $\partial\Omega$ for the Neumann problem requires additional boundary-parametrix analysis (with the precise coefficient of the local logarithmic singularity at boundary points typically modified by reflected-image contributions). All concentration computations below are stated for interior $x_0$.
\end{remark}

\section{Free energy, dissipation, and concentration scaling}
\label{sec:PES-energy}

We now exploit the operator-theoretic structure of $K_\tau$ to identify the gradient-flow structure of PES, derive the entropy-dissipation identity, and compute the concentration scaling that singles out $M_{\!*}=64\pi^{2}\tau/\chi$ as the candidate critical mass in dimension four.

\subsection{Eliminated PES equation as a gradient flow}
\label{subsec:PES-gradient-flow}

Using $c=K_\tau u$, the first equation of \eqref{eq:PES-intro} becomes
\begin{equation}\label{eq:PES-eliminated}
        \partial_t u=\Delta u-\chi\Div(u\nabla K_\tau u)
        =\Div\!\bigl[u\,\nabla(\log u-\chi K_\tau u)\bigr],
\end{equation}
where the second equality uses $\Delta u=\Div(u\nabla\log u)$ for positive smooth $u$. The natural free energy associated with \eqref{eq:PES-eliminated} is
\begin{equation}\label{eq:PES-free-energy}
        \Free_{\mathrm{PES}}[u]
        =\int_\Omega u\log u\,dx
        -\frac{\chi}{2}\int_\Omega u\,K_\tau u\,dx.
\end{equation}
Self-adjointness of $K_\tau$ is essential: for any test function $\phi$,
\begin{equation}\label{eq:variation-Ktau}
        \frac{d}{d\varepsilon}\bigg|_{\varepsilon=0}\frac{1}{2}\int_\Omega(u+\varepsilon\phi)K_\tau(u+\varepsilon\phi)\,dx
        =\int_\Omega\phi\,K_\tau u\,dx,
\end{equation}
so that the $L^{2}$-gradient of $\Free_{\mathrm{PES}}$ is
\begin{equation}\label{eq:PES-first-variation}
        \frac{\delta\Free_{\mathrm{PES}}}{\delta u}(u)=\log u+1-\chi K_\tau u,
\end{equation}
and equation \eqref{eq:PES-eliminated} is the formal $L^{2}$-gradient flow $\partial_t u=\Div(u\nabla\tfrac{\delta\Free_{\mathrm{PES}}}{\delta u})$ in the Wasserstein sense; cf.~\cite{jordan1998variational, ambrosio2005gradient} for the abstract setting.

\begin{proposition}[Entropy dissipation]\label{prop:PES-dissipation}
Let $u\in C([0,T];L^{1}\cap L^{\infty}(\Omega))\cap C^{1}((0,T);L^{1}(\Omega))$ be a positive, sufficiently smooth solution of \eqref{eq:PES-eliminated} satisfying one of the following: (a) periodic boundary conditions; (b) Schwartz decay at infinity on $\R^{N}$; (c) the no-flux boundary condition
\begin{equation}\label{eq:no-flux}
        \bigl(\nabla u-\chi u\,\nabla K_\tau u\bigr)\cdot\nu=0
        \qquad\text{on }\partial\Omega.
\end{equation}
Then
\begin{equation}\label{eq:PES-dissipation}
        \frac{d}{dt}\Free_{\mathrm{PES}}[u(t)]
        =-\int_\Omega u\,\bigl|\nabla(\log u-\chi K_\tau u)\bigr|^{2}\,dx\le 0,
\end{equation}
and in integrated form,
\begin{equation}\label{eq:PES-integrated-dissipation}
        \Free_{\mathrm{PES}}[u(T)]+\int_0^T\!\!\int_\Omega u\,\bigl|\nabla(\log u-\chi K_\tau u)\bigr|^{2}\,dx\,dt
        =\Free_{\mathrm{PES}}[u_0].
\end{equation}
In addition, $\int_\Omega u(\cdot,t)\,dx$ is conserved.
\end{proposition}

\begin{proof}
By \eqref{eq:PES-first-variation},
\[
\begin{aligned}
        \frac{d}{dt}\Free_{\mathrm{PES}}[u(t)]
        &=\int_\Omega(\log u-\chi K_\tau u)\,\partial_t u\,dx\\
        &=\int_\Omega(\log u-\chi K_\tau u)\,\Div\bigl[u\nabla(\log u-\chi K_\tau u)\bigr]\,dx,
\end{aligned}
\]
the additive constant from \eqref{eq:PES-first-variation} contributing nothing because $\partial_t\int u\,dx=0$ (which follows from \eqref{eq:no-flux} or the analogous decay/periodicity hypothesis). Integration by parts gives \eqref{eq:PES-dissipation} with vanishing boundary term. Integration in time gives \eqref{eq:PES-integrated-dissipation}. Mass conservation is immediate from $\partial_t u=\Div(\cdots)$ together with the boundary/decay hypothesis.
\end{proof}

\subsection{Entropy and interaction scaling of concentrating sequences}
\label{subsec:concentration-scaling}

Throughout this subsection, $N=4$. Fix an interior point $x_0\in\Omega$ with $\dist(x_0,\partial\Omega)>0$ (or $x_0\in\R^{4}$, resp.\ $x_0\in\T^{4}$). Fix $U\in C_c^{\infty}(B_{1}(0))$ with $U\ge 0$ and $\int_{\R^{4}}U(z)\,dz=M$, and define
\begin{equation}\label{eq:u-eps}
        u_\varepsilon(x):=\varepsilon^{-4}U\!\left(\frac{x-x_0}{\varepsilon}\right),
        \qquad 0<\varepsilon\ll 1.
\end{equation}
Then $\int_\Omega u_\varepsilon\,dx=M$ and $u_\varepsilon\,dx\rightharpoonup M\delta_{x_0}$ weakly-$*$ as measures.

\begin{lemma}[Entropy scaling]\label{lem:entropy-scaling}
For $u_\varepsilon$ as in \eqref{eq:u-eps},
\begin{equation}\label{eq:entropy-scaling}
        \int_\Omega u_\varepsilon\log u_\varepsilon\,dx
        =4M\log\frac{1}{\varepsilon}+\int_{\R^{4}}U\log U\,dz
        =4M\log\frac{1}{\varepsilon}+O(1).
\end{equation}
\end{lemma}

\begin{proof}
Substitute $x=x_0+\varepsilon z$, $dx=\varepsilon^{4}dz$. Then $u_\varepsilon(x)=\varepsilon^{-4}U(z)$ and
\[
        \int_\Omega u_\varepsilon\log u_\varepsilon\,dx
        =\int_{\R^{4}}U(z)\bigl(4\log\varepsilon^{-1}+\log U(z)\bigr)\,dz
        =4M\log\varepsilon^{-1}+\int_{\R^{4}}U\log U\,dz.\qedhere
\]
\end{proof}

\begin{lemma}[Interaction scaling]\label{lem:interaction-scaling}
For $u_\varepsilon$ as in \eqref{eq:u-eps},
\begin{equation}\label{eq:interaction-scaling}
        \int_\Omega u_\varepsilon\,K_\tau u_\varepsilon\,dx
        =\frac{M^{2}}{8\pi^{2}\tau}\log\frac{1}{\varepsilon}+O(1).
\end{equation}
\end{lemma}

\begin{proof}
By Proposition~\ref{prop:log-kernel-R4}, for $x,y$ in the compact support $\mathcal K\Subset\Omega$ of $u_\varepsilon$ (uniformly in $\varepsilon$ small),
\[
        G_\tau(x,y)=\frac{1}{8\pi^{2}\tau}\log\frac{1}{|x-y|}+R_\tau(x,y),
        \qquad
        R_\tau\in L^{\infty}(\mathcal K\times\mathcal K).
\]
Hence
\[
\begin{aligned}
        \int_\Omega u_\varepsilon K_\tau u_\varepsilon\,dx
        &=\frac{1}{8\pi^{2}\tau}\iint u_\varepsilon(x)u_\varepsilon(y)\log\frac{1}{|x-y|}\,dx\,dy
        +\iint u_\varepsilon(x)R_\tau(x,y)u_\varepsilon(y)\,dx\,dy.
\end{aligned}
\]
The remainder integral is bounded by $\|R_\tau\|_{L^{\infty}(\mathcal K\times\mathcal K)}M^{2}=O(1)$. For the logarithmic part, substitute $x=x_0+\varepsilon z$, $y=x_0+\varepsilon\zeta$:
\[
\begin{aligned}
        &\iint u_\varepsilon(x)u_\varepsilon(y)\log\frac{1}{|x-y|}\,dx\,dy\\
        &\qquad=\iint_{\R^{4}\times\R^{4}}U(z)U(\zeta)\Bigl(\log\varepsilon^{-1}+\log\frac{1}{|z-\zeta|}\Bigr)\,dz\,d\zeta\\
        &\qquad=M^{2}\log\varepsilon^{-1}+\iint U(z)U(\zeta)\log\frac{1}{|z-\zeta|}\,dz\,d\zeta.
\end{aligned}
\]
The last integral is finite since $U$ is smooth and compactly supported, so \eqref{eq:interaction-scaling} follows.
\end{proof}

\begin{proposition}[Concentration-scaling candidate mass]\label{prop:candidate-mass}
Let $N=4$ and let $u_\varepsilon$ be as in \eqref{eq:u-eps}. Then
\begin{equation}\label{eq:F-concentration}
        \Free_{\mathrm{PES}}[u_\varepsilon]
        =\left(4M-\frac{\chi M^{2}}{16\pi^{2}\tau}\right)\log\frac{1}{\varepsilon}+O(1).
\end{equation}
The leading coefficient vanishes precisely at
\begin{equation}\label{eq:M-star-final}
        M_{\!*}=\frac{64\pi^{2}\tau}{\chi}.
\end{equation}
Consequently:
\begin{enumerate}[label=(\roman*),leftmargin=2em]
\item if $M<M_{\!*}$, then $\Free_{\mathrm{PES}}[u_\varepsilon]\to +\infty$ as $\varepsilon\downarrow 0$;
\item if $M=M_{\!*}$, then $\Free_{\mathrm{PES}}[u_\varepsilon]=O(1)$ as $\varepsilon\downarrow 0$;
\item if $M>M_{\!*}$, then $\Free_{\mathrm{PES}}[u_\varepsilon]\to -\infty$ as $\varepsilon\downarrow 0$. (see the variational descent illustrated in Figure~\ref{fig:free_energy_landscape}).
\end{enumerate}
\end{proposition}

\begin{proof}
Substitute \eqref{eq:entropy-scaling} and \eqref{eq:interaction-scaling} into \eqref{eq:PES-free-energy}:
\[
\begin{aligned}
        \Free_{\mathrm{PES}}[u_\varepsilon]
        &=4M\log\varepsilon^{-1}-\frac{\chi}{2}\cdot\frac{M^{2}}{8\pi^{2}\tau}\log\varepsilon^{-1}+O(1)\\
        &=\Bigl(4M-\frac{\chi M^{2}}{16\pi^{2}\tau}\Bigr)\log\varepsilon^{-1}+O(1).
\end{aligned}
\]
The leading coefficient $4M-\chi M^{2}/(16\pi^{2}\tau)=M\bigl(4-\chi M/(16\pi^{2}\tau)\bigr)$ vanishes at $M=64\pi^{2}\tau/\chi$, and the three cases follow from the sign of the coefficient.
\end{proof}

\begin{corollary}[Variational descent above the candidate mass]\label{cor:descent-above-M-star}
If $M>64\pi^{2}\tau/\chi$, then there exist $u_\varepsilon$ concentrating at any interior point $x_0$ such that $u_\varepsilon\,dx\rightharpoonup M\delta_{x_0}$ and $\Free_{\mathrm{PES}}[u_\varepsilon]\to -\infty$ as $\varepsilon\downarrow 0$.\qed
\end{corollary}

\begin{remark}[The threshold value is variational, not dynamical]\label{rem:not-dynamical}
Corollary~\ref{cor:descent-above-M-star} produces sequences along which the free energy is unbounded below. By itself this does \emph{not} imply finite-time blow-up of the PES flow: the entropy-dissipation identity \eqref{eq:PES-dissipation} only forces $t\mapsto\Free_{\mathrm{PES}}[u(t)]$ to be non-increasing. Promoting variational descent to dynamical concentration requires further dynamical tools (virial identity, monotonicity formula, or a sharp $K_\tau$-adapted compactness theory). We formulate the resulting open problem in Section~\ref{sec:open}.
\end{remark}

\begin{figure}[!h]
    \centering
    \includegraphics[width=0.85\textwidth]{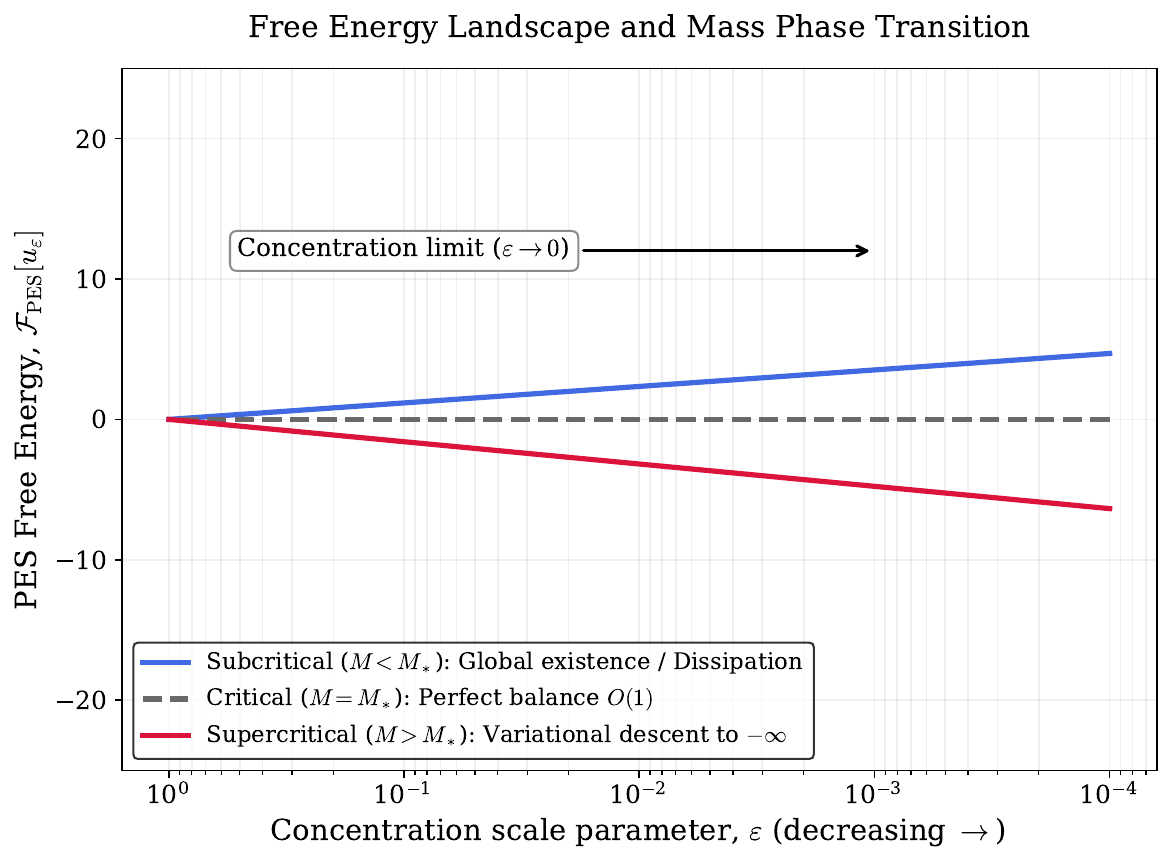}
    \caption{Free energy landscape and mass phase transition for the PES cascade. The scaling of the free energy $\mathcal{F}_{\mathrm{PES}}[u_\varepsilon]$ is plotted against the concentration parameter $\varepsilon$. Subcritical masses ($M < M_*$) face an infinite energy barrier as $\varepsilon \to 0$, ensuring dissipation and preventing finite-time condensation. Supercritical masses ($M > M_*$) exhibit an unbounded variational descent to $-\infty$, indicating a catastrophic loss of compactness. The threshold $M_* = 64\pi^2\tau/\chi$ marks the exact balance between random diffusion and fourth-order spatial aggregation.}
    \label{fig:free_energy_landscape}
\end{figure}

\section{Volterra-memory structure of MEP}
\label{sec:MEP}

We turn to the mixed elliptic--parabolic cascade \eqref{eq:MEP-intro}. The intermediate elliptic equation $0=\Delta w-w+u$ gives $w=A^{-1}u$ at each time, so that the signal equation reduces to
\begin{equation}\label{eq:MEP-c-equation}
        \partial_t c+Ac=A^{-1}u.
\end{equation}

\subsection{Duhamel representation and Volterra drift}
\label{subsec:Duhamel}

Let $\{e^{-tA}\}_{t\ge 0}$ denote the analytic semigroup generated by $-A$.

\begin{proposition}[Volterra representation]\label{prop:MEP-Volterra}
Let $(u,c)$ be sufficiently smooth so that all of the operations below are justified. Then
\begin{equation}\label{eq:MEP-Duhamel}
        c(t)=e^{-tA}c_0+\int_0^t e^{-(t-s)A}A^{-1}u(s)\,ds,
\end{equation}
and consequently
\begin{equation}\label{eq:nabla-c-Volterra}
        \nabla c(t)=\nabla e^{-tA}c_0+\Vop u(t),
        \quad
        \Vop u(t):=\int_0^t \nabla e^{-(t-s)A}A^{-1}u(s)\,ds.
\end{equation}
\end{proposition}

\begin{proof}
Apply $e^{tA}$ to \eqref{eq:MEP-c-equation} and integrate from $0$ to $t$:
\[
        e^{tA}c(t)-c_0=\int_0^t e^{sA}A^{-1}u(s)\,ds.
\]
Multiplication by $e^{-tA}$ gives \eqref{eq:MEP-Duhamel}. Differentiation in $x$ yields \eqref{eq:nabla-c-Volterra}.
\end{proof}

\subsection{Drift multiplier: near-diagonal versus time-averaged orders}
\label{subsec:multiplier}

On $\R^{N}$, $\widehat{e^{-\theta A}f}(\xi)=e^{-\theta(1+|\xi|^{2})}\widehat f(\xi)$ and $\widehat{A^{-1}f}(\xi)=(1+|\xi|^{2})^{-1}\widehat f(\xi)$, so the memory drift multiplier at time lag $\theta=t-s>0$ is
\begin{equation}\label{eq:m-xi-theta}
        m(\xi,\theta)=\frac{i\xi\,e^{-\theta(1+|\xi|^{2})}}{1+|\xi|^{2}}.
\end{equation}

\begin{proposition}[Near-diagonal Keller--Segel order]\label{prop:MEP-near-diagonal}
For every fixed $\theta>0$, $|m(\xi,\theta)|\le |\xi|(1+|\xi|^{2})^{-1}e^{-\theta|\xi|^{2}}$ decays faster than any polynomial as $|\xi|\to\infty$, so $\mathcal{F}^{-1}m(\cdot,\theta)$ is a smoothing operator of infinite order. However, the family $\{m(\cdot,\theta)\}_{\theta>0}$ is \emph{not} uniformly smoothing as $\theta\downarrow 0$: 
\begin{equation}\label{eq:m-near-diagonal}
        m(\xi,\theta)\longrightarrow m_0(\xi):=\frac{i\xi}{1+|\xi|^{2}}
        \qquad\text{locally uniformly in }\xi\text{ as }\theta\downarrow 0,
\end{equation}
and the limiting multiplier $m_0$ has principal order $-1$, equal to that of the classical Keller--Segel drift $\nabla(I-\Delta)^{-1}$. Thus the symbol-class order of the Volterra-drift kernel at the time diagonal is Keller--Segel-like, even though no individual time-slice operator $\mathcal{F}^{-1}m(\cdot,\theta)$ (with $\theta>0$) has order $-1$.
\end{proposition}

\begin{proof}
The first claim follows from $e^{-\theta(1+|\xi|^{2})}\le e^{-\theta|\xi|^{2}}$, and the limit \eqref{eq:m-near-diagonal} from $e^{-\theta(1+|\xi|^{2})}\to 1$ as $\theta\downarrow 0$, uniformly on compact $\xi$-sets. For $|\xi|\to\infty$, $|m_0(\xi)|=|\xi|/(1+|\xi|^{2})\sim|\xi|^{-1}$, so $m_0$ has principal order $-1$. The non-uniformity in $\theta$ follows because the symbol estimate $|m(\xi,\theta)|\le|\xi|^{-1}$ at fixed $\theta=0^+$ does not improve to higher decay even as $|\xi|\to\infty$.
\end{proof}

The following observation isolates the second competing feature of the MEP drift.

\begin{proposition}[Frozen-time average is fourth-order smoothing]\label{prop:frozen-average}
If $u(s)\equiv u$ is frozen in time, then
\begin{equation}\label{eq:frozen-average}
        \int_0^t \nabla e^{-(t-s)A}A^{-1}u\,ds
        =\nabla A^{-2}\bigl(I-e^{-tA}\bigr)u.
\end{equation}
The associated Fourier multiplier
\[
        M_t(\xi)=\frac{i\xi\bigl(1-e^{-t(1+|\xi|^{2})}\bigr)}{(1+|\xi|^{2})^{2}}
\]
satisfies $M_t(\xi)\sim i\xi/|\xi|^{4}$ as $|\xi|\to\infty$ for each fixed $t>0$, and therefore has principal order $-3$.
\end{proposition}

\begin{proof}
By change of variables $\theta=t-s$,
\[
        \int_0^t \nabla e^{-(t-s)A}A^{-1}u\,ds
        =\nabla A^{-1}\int_0^t e^{-\theta A}u\,d\theta
        =\nabla A^{-1}\cdot A^{-1}(I-e^{-tA})u
        =\nabla A^{-2}(I-e^{-tA})u,
\]
where we used $\int_0^t e^{-\theta A}\,d\theta=A^{-1}(I-e^{-tA})$. The Fourier multiplier follows by direct computation, and the high-frequency asymptotics yield principal order $-3$.
\end{proof}

\begin{remark}[Time-averaged smoothing is not a static reduction]\label{rem:not-static-reduction}
For nonlinear MEP solutions, $u$ is \emph{not} frozen in time. The nonlinear memory term
\[
        \Vop u(t)=\int_0^t \nabla e^{-(t-s)A}A^{-1}u(s)\,ds
\]
cannot be replaced by $\nabla A^{-2}(I-e^{-tA})u(t)$ without an additional approximation argument controlling the time variation of $u(s)$ on the interval $[0,t]$. Consequently, the MEP criticality problem is a genuinely mixed space--time problem: its near-diagonal component (Proposition~\ref{prop:MEP-near-diagonal}) has Keller--Segel order $-1$, while its frozen-time average (Proposition~\ref{prop:frozen-average}) has fourth-order order $-3$. The static fourth-order elliptic scaling of PES cannot be transferred to MEP through the frozen-time average alone.
\end{remark}

\subsection{Space--time Volterra estimates}
\label{subsec:Volterra-estimates}

Although the near-diagonal multiplier has Keller--Segel order $-1$, the Volterra structure provides quantitative smoothing measured in mixed space--time norms. We record the basic estimates.

\begin{lemma}[Heat--Bessel estimate]\label{lem:heat-Bessel}
Let $1<q\le s<\infty$ and $\theta>0$. There exists $C=C(N,q,s)$ such that
\begin{equation}\label{eq:heat-Bessel}
        \|\nabla e^{-\theta A}A^{-1}f\|_{L^{s}(\R^{N})}
        \le C\,e^{-\theta}\theta^{-\beta(q,s)}\|f\|_{L^{q}(\R^{N})},
\end{equation}
with
\begin{equation}\label{eq:beta-qs}
        \beta(q,s):=\Bigl(\tfrac{N}{2}\bigl(\tfrac{1}{q}-\tfrac{1}{s}\bigr)-\tfrac{1}{2}\Bigr)_{+}.
\end{equation}
The analogous estimate (with $e^{-\theta}$ replaced by $1$) holds on $\T^{N}$ and on bounded Neumann domains.
\end{lemma}

\begin{proof}
The Fourier multiplier of $\nabla A^{-1}$ is $i\xi/(1+|\xi|^{2})$, which satisfies the Mikhlin--Hörmander condition $|\xi|^{|\alpha|}|\partial_\xi^{\alpha}(i\xi/(1+|\xi|^{2}))|\le C_\alpha$ for every multi-index $\alpha$. Hence $\nabla A^{-1}$ is bounded on $L^{p}(\R^{N})$ for $1<p<\infty$, and by Bessel-potential / Sobolev embedding it maps $L^{q}\to L^{r}$ boundedly for any $1<q\le r<\infty$ satisfying
\[
        \frac{1}{r}\ge\frac{1}{q}-\frac{1}{N}.
\]
The heat semigroup $\{e^{\theta\Delta}\}_{\theta>0}$ satisfies the standard ultracontractive estimate
\[
        \|e^{\theta\Delta}g\|_{L^{s}(\R^{N})}\le C\theta^{-\frac{N}{2}(\frac{1}{r}-\frac{1}{s})}\|g\|_{L^{r}(\R^{N})},
        \qquad 1\le r\le s\le\infty.
\]
Writing $e^{-\theta A}=e^{-\theta}e^{\theta\Delta}$, the Bessel mass term contributes the exponential factor $e^{-\theta}$, and the heat factor contributes the ultracontractive gain in integrability.

\emph{Case $1/q-1/s\le 1/N$.}
In this range, $\nabla A^{-1}:L^{q}\to L^{s}$ is bounded directly, and $\|e^{-\theta A}h\|_{L^{s}}\le e^{-\theta}\|h\|_{L^{s}}$. Hence
\[
        \|\nabla e^{-\theta A}A^{-1}f\|_{L^{s}}
        =\|e^{-\theta A}\nabla A^{-1}f\|_{L^{s}}
        \le e^{-\theta}\|\nabla A^{-1}f\|_{L^{s}}
        \le Ce^{-\theta}\|f\|_{L^{q}},
\]
which is \eqref{eq:heat-Bessel} with $\beta(q,s)=0$.

\emph{Case $1/q-1/s>1/N$.}
Set $1/r:=1/q-1/N$, so that $\nabla A^{-1}:L^{q}\to L^{r}$ boundedly. Then
\[
        \|\nabla e^{-\theta A}A^{-1}f\|_{L^{s}}
        \le \|e^{-\theta A}\|_{L^{r}\to L^{s}}\,\|\nabla A^{-1}f\|_{L^{r}}
        \le Ce^{-\theta}\theta^{-\frac{N}{2}(\frac{1}{r}-\frac{1}{s})}\|f\|_{L^{q}}.
\]
By the choice of $r$, the exponent is
\[
        \frac{N}{2}\!\left(\frac{1}{r}-\frac{1}{s}\right)
        =\frac{N}{2}\!\left(\frac{1}{q}-\frac{1}{N}-\frac{1}{s}\right)
        =\frac{N}{2}\!\left(\frac{1}{q}-\frac{1}{s}\right)-\frac{1}{2}=\beta(q,s),
\]
yielding \eqref{eq:heat-Bessel}.

The torus and bounded Neumann domain cases proceed identically, using the Mikhlin--Hörmander theorem in those geometries (with the exponential decay $e^{-\theta}$ absent on bounded domains, where it is absorbed into a constant on $[0,T]$).
\end{proof}

\begin{lemma}[Pointwise-in-time Volterra bound]\label{lem:Volterra-pointwise}
Let $T>0$, $1<q\le s<\infty$. For $0<t\le T$,
\begin{equation}\label{eq:Volterra-pointwise}
        \|\Vop u(t)\|_{L^{s}}
        \le C\int_0^t (t-\sigma)^{-\beta(q,s)}\|u(\sigma)\|_{L^{q}}\,d\sigma.
\end{equation}
\end{lemma}

\begin{proof}
Apply Lemma~\ref{lem:heat-Bessel} pointwise in $\sigma$ to the integrand in $\Vop u(t)=\int_0^t \nabla e^{-(t-\sigma)A}A^{-1}u(\sigma)\,d\sigma$ and integrate.
\end{proof}

\begin{proposition}[Mixed-norm Volterra estimate]\label{prop:mixed-norm-Volterra}
Let $T>0$, $1<q\le s<\infty$, and $\beta=\beta(q,s)$ as in \eqref{eq:beta-qs}.
\begin{enumerate}[label=(\roman*),leftmargin=2em]
\item Suppose $\beta>0$. Let $1\le p,r,a\le\infty$ satisfy
\begin{equation}\label{eq:Young-condition}
        1+\frac{1}{r}=\frac{1}{a}+\frac{1}{p},\qquad a\beta<1.
\end{equation}
Then
\begin{equation}\label{eq:mixed-norm-Volterra}
        \|\Vop u\|_{L^{r}(0,T;L^{s})}
        \le C\,T^{1/a-\beta}\,\|u\|_{L^{p}(0,T;L^{q})}.
\end{equation}
\item Suppose $\beta=0$. Then for every $1\le p\le\infty$,
\begin{equation}\label{eq:mixed-norm-beta-zero}
        \|\Vop u\|_{L^{p}(0,T;L^{s})}\le C\,T\,\|u\|_{L^{p}(0,T;L^{q})}.
\end{equation}
\end{enumerate}
\end{proposition}

\begin{proof}
By Lemma~\ref{lem:Volterra-pointwise}, $\|\Vop u(t)\|_{L^{s}}\le C(k_\beta\star\|u\|_{L^{q}})(t)$, where $k_\beta(\theta)=\theta^{-\beta}\mathbf 1_{(0,T)}(\theta)$ and $\star$ denotes convolution on the half-line (extending by zero outside $(0,T)$).

For $\beta>0$ and $a\beta<1$, $k_\beta\in L^{a}(0,T)$ with $\|k_\beta\|_{L^{a}(0,T)}=((1-a\beta)^{-1}T^{1-a\beta})^{1/a}\le C_{\beta,a}T^{1/a-\beta}$. Young's convolution inequality \eqref{eq:Young-condition} gives
$
        \|\Vop u\|_{L^{r}(0,T;L^{s})}\le C\|k_\beta\|_{L^{a}(0,T)}\|u\|_{L^{p}(0,T;L^{q})},
$
yielding \eqref{eq:mixed-norm-Volterra}. For $\beta=0$, \eqref{eq:Volterra-pointwise} gives $\|\Vop u(t)\|_{L^{s}}\le C\int_0^t\|u(\sigma)\|_{L^{q}}\,d\sigma\le C\,T^{1-1/p}\|u\|_{L^{p}(0,T;L^{q})}$ (Hölder), and \eqref{eq:mixed-norm-beta-zero} follows.
\end{proof}

\begin{theorem}[MEP drift estimate]\label{thm:MEP-drift}
Let $T>0$, $1<q\le s<\infty$, and let exponents satisfy the hypotheses of Proposition~\ref{prop:mixed-norm-Volterra}. Suppose $c_0\in W^{1,s}(\Omega)$ and $u\in L^{p}(0,T;L^{q}(\Omega))$. Then the MEP signal satisfies
\begin{equation}\label{eq:MEP-drift-bound}
        \|\nabla c\|_{L^{r}(0,T;L^{s})}
        \le CT^{1/r}\|\nabla c_0\|_{L^{s}}+CT^{1/a-\beta(q,s)}\|u\|_{L^{p}(0,T;L^{q})}
\end{equation}
when $\beta(q,s)>0$. If $\beta(q,s)=0$, the second term is replaced by $CT\|u\|_{L^{p}(0,T;L^{q})}$ with $r=p$.
\end{theorem}

\begin{proof}
From \eqref{eq:nabla-c-Volterra}, $\|\nabla c(t)\|_{L^{s}}\le \|\nabla e^{-tA}c_0\|_{L^{s}}+\|\Vop u(t)\|_{L^{s}}$. The semigroup bound $\|\nabla e^{-tA}c_0\|_{L^{s}}\le C\|\nabla c_0\|_{L^{s}}$ gives the first term of \eqref{eq:MEP-drift-bound}; the second is Proposition~\ref{prop:mixed-norm-Volterra}.
\end{proof}

\begin{remark}[Space--time, not static, criticality]\label{rem:spacetime-criticality}
Theorem~\ref{thm:MEP-drift} is the correct form of estimate for MEP: smoothing appears through a time convolution. Importing the PES critical exponent $q_c=N/4$ from Proposition~\ref{prop:PES-scaling} is not legitimate without additional control on the time variation of $u$. Any threshold theory for MEP must be expressed in mixed norms.
\end{remark}

\subsection{Mild formulation and continuation criterion}
\label{subsec:mild-continuation}

The density equation in MEP, written in divergence form $\partial_t u=\Delta u-\chi\Div(u\nabla c)$, admits the mild formulation
\begin{equation}\label{eq:MEP-mild-u}
        u(t)=e^{t\Delta}u_0
        -\chi\int_0^t \nabla e^{(t-s)\Delta}\cdot\bigl(u(s)\nabla c(s)\bigr)\,ds,
\end{equation}
coupled to $\nabla c=\nabla e^{-tA}c_0+\Vop u$ via Proposition~\ref{prop:MEP-Volterra}.

\begin{proposition}[Classical continuation criterion]\label{prop:MEP-continuation}
Let $u_0\in L^{1}(\Omega)\cap L^{\infty}(\Omega)$ with $u_0\ge 0$ and $c_0\in W^{1,\infty}(\Omega)$. Let $(u,c)$ be a classical solution of \eqref{eq:MEP-intro} on $[0,T_{\max})$, $T_{\max}\in (0,\infty]$. If $T<T_{\max}$ and
\begin{equation}\label{eq:continuation-bound}
        \|u\|_{L^{\infty}(0,T;L^{\infty})}+\|\nabla c\|_{L^{1}(0,T;W^{1,\infty})}<\infty,
\end{equation}
then the solution extends past $T$. In particular, if $T_{\max}<\infty$, then
\begin{equation}\label{eq:blowup-alternative}
        \limsup_{T\uparrow T_{\max}}\Bigl[\|u\|_{L^{\infty}(0,T;L^{\infty})}+\|\nabla c\|_{L^{1}(0,T;W^{1,\infty})}\Bigr]=\infty.
\end{equation}
\end{proposition}

\begin{proof}
The hypothesis \eqref{eq:continuation-bound} ensures that the drift $\nabla c\in L^{1}(0,T;W^{1,\infty})$ and $u\in L^{\infty}(0,T;L^{\infty})$, so the density equation $\partial_t u-\Delta u=-\chi\Div(u\nabla c)$ is a uniformly parabolic equation with controlled lower-order coefficients. Standard parabolic regularity (cf.~\cite{ladyzhenskaia1968linear}) yields Hölder estimates and then classical bounds for $u$ up to time $T$, depending only on the quantities in \eqref{eq:continuation-bound} and the initial data. The signal equation $\partial_t c+Ac=A^{-1}u$ is linear parabolic with $L^{\infty}$ right-hand side, so $c$ acquires corresponding regularity. The classical local existence theorem then extends the solution past $T$.
\end{proof}

\begin{remark}[No three-dimensional mass-critical theorem]\label{rem:no-3D-MEP}
The estimates above do \emph{not} imply a three-dimensional mass-critical theorem for MEP. The near-diagonal multiplier has Keller--Segel order $-1$, whereas the frozen-time average has fourth-order order $-3$. Whether the Volterra memory produces a new effective critical threshold, distinct from both the Keller--Segel and PES thresholds, requires a scale-invariant mixed-norm theory of the nonlinear map
\[
        u\longmapsto\Div\!\left(u\int_0^t\nabla e^{-(t-s)A}A^{-1}u(s)\,ds\right).
\]
This question is formulated as Problem~\ref{prob:MEP-criticality} in Section~\ref{sec:open}.
\end{remark}

\section{Open threshold problems and conditional statements}
\label{sec:open}

The structural results of the previous sections settle the operator classification but leave open the corresponding sharp threshold theory. We record here the precise problems whose resolution would complete the picture, together with conditional statements that follow immediately from the structural analysis.

\subsection{Sharp $K_\tau$-adapted Adams/log-HLS inequality for PES}
\label{subsec:sharp-log-HLS}

In view of the four-dimensional logarithmic kernel (Proposition~\ref{prop:log-kernel-R4}) and the concentration scaling (Proposition~\ref{prop:candidate-mass}), the natural sharp inequality is the fourth-order Adams/log-HLS analogue of the two-dimensional Beckner log-HLS inequality \cite{beckner1993sharp}. Let
\begin{equation}\label{eq:admissible-class}
        \Adm_M(\Omega):=\left\{u\ge 0:\ \int_\Omega u\,dx=M,\ \int_\Omega u(1+|\log u|)\,dx<\infty\right\}.
\end{equation}
The condition $\int u|\log u|<\infty$ ensures that the entropy $\int u\log u$ is finite \emph{and} that the inequality \eqref{eq:target-inequality} below is non-vacuous: $\int u\log u$ can be negative, so the inequality is meaningful only when the entropy side is bounded below. On a bounded domain $\Omega$ this is automatic from the elementary bound $u\log u\ge -1/e$. On $\R^{4}$, finiteness must be supplemented by a tightness or moment condition such as $\int|x|^{2}u\,dx<\infty$ to ensure that $\int u\log u$ is bounded below on $\Adm_M$; this is standard in two-dimensional Keller--Segel theory \cite{blanchet2006two}. We adopt this convention throughout.

\begin{problem}[Sharp $K_\tau$-adapted log-HLS/Adams inequality]\label{prob:sharp-logHLS}
Let $N=4$ and let $\Omega$ be either a bounded $C^{\infty}$ Neumann domain, $\T^{4}$, or $\R^{4}$ (in the last case with the moment normalisation above). Determine the sharp constants for which
\begin{equation}\label{eq:target-inequality}
        \int_\Omega u\,K_\tau u\,dx\le a_{\sharp}(M,\tau,\Omega)\int_\Omega u\log u\,dx+C(M,\tau,\Omega)
\end{equation}
holds for all $u\in\Adm_M(\Omega)$, with $a_{\sharp}$ and $C$ depending only on $M$, $\tau$, and $\Omega$ (and on the boundary condition). Since $\int u\log u$ may take both signs, the meaningful subcritical condition is $\tfrac{\chi}{2}a_{\sharp}<1$, which by Proposition~\ref{prop:conditional-PES} converts the inequality into a bound on the free energy modulo a finite additive constant.
\end{problem}

The one-point concentration computation forces any sharp inequality of the form \eqref{eq:target-inequality} to satisfy the lower constraint
\begin{equation}\label{eq:a-sharp-forced}
        a_{\sharp}(M,\tau,\Omega)\ge\frac{M}{32\pi^{2}\tau}
\end{equation}
at the logarithmic concentration scale. Indeed, for the concentrating family $u_\varepsilon$ of \eqref{eq:u-eps}, Lemmas~\ref{lem:entropy-scaling}--\ref{lem:interaction-scaling} give
\[
        \frac{\int_\Omega u_\varepsilon\,K_\tau u_\varepsilon\,dx}{\int_\Omega u_\varepsilon\log u_\varepsilon\,dx}
        =\frac{M}{32\pi^{2}\tau}+o(1),\qquad \varepsilon\downarrow 0,
\]
so any constant $a$ for which \eqref{eq:target-inequality} holds uniformly along this family must satisfy $a\ge M/(32\pi^{2}\tau)$, with equality forced if the inequality is to be saturated by one-point concentration. The natural conjectural sharp form is therefore
\begin{equation}\label{eq:conjectural-a-sharp}
        a_{\sharp}(M,\tau,\Omega)=\frac{M}{32\pi^{2}\tau},
\end{equation}
up to lower-order domain-dependent corrections absorbed into $C(M,\tau,\Omega)$. With this value, the subcritical condition $\tfrac{\chi}{2}a_{\sharp}<1$ becomes exactly
\[
        M<\frac{64\pi^{2}\tau}{\chi}=M_{\!*},
\]
matching the candidate critical mass of Proposition~\ref{prop:candidate-mass}. This consistency is the heuristic basis for the following conjecture.

\begin{conjecture}[Sharp threshold for PES free energy in $N=4$]\label{conj:sharp-threshold}
Let $N=4$ and let $\Omega$ be $\R^{4}$, $\T^{4}$, or a bounded $C^{2}$ Neumann domain. The PES free energy $\Free_{\mathrm{PES}}$ satisfies the alternative:
\begin{enumerate}[label=(\roman*),leftmargin=2em]
\item if $0<M<M_{\!*}$, then $\inf_{u\in\Adm_M(\Omega)}\Free_{\mathrm{PES}}[u]>-\infty$ \emph{(}with the usual tightness or moment condition on $\R^{4}$ to control the entropy from below\emph{)};
\item if $M>M_{\!*}$, then $\inf_{u\in\Adm_M(\Omega)}\Free_{\mathrm{PES}}[u]=-\infty$;
\item at $M=M_{\!*}$, loss of compactness occurs through one-point mass concentration, with the precise compactness alternative depending on the domain and on the boundary conditions.
\end{enumerate}
\end{conjecture}

Item (ii) is established (Corollary~\ref{cor:descent-above-M-star}). The crux of the conjecture is item (i), which is equivalent to Problem~\ref{prob:sharp-logHLS} together with the appropriate Cauchy--Schwarz/Jensen analysis. The fourth-order Adams inequality of \cite{adams1988sharp} and its sharp form in $\R^{4}$ \cite{lam2012sharp} provide an exponential integrability framework but do not directly imply a sharp log-HLS form adapted to the specific kernel $K_\tau$; the natural conjectural form is dictated by \eqref{eq:conjectural-a-sharp}.

\subsection{Supercritical PES dynamics}
\label{subsec:supercritical-dynamics}

For $M>M_{\!*}$, Corollary~\ref{cor:descent-above-M-star} produces sequences with $\Free_{\mathrm{PES}}[u_\varepsilon]\to -\infty$. The entropy-dissipation identity \eqref{eq:PES-dissipation} implies that smooth PES solutions cannot increase their free energy, but this does not by itself yield dynamical concentration.

\begin{problem}[Supercritical PES dynamics in $N=4$]\label{prob:PES-dynamics}
Let $N=4$ and $M>M_{\!*}$. Determine:
\begin{enumerate}[label=(\roman*),leftmargin=2em]
\item whether negative free energy implies finite-time blow-up of PES;
\item if blow-up occurs, whether the concentrating profile is governed by the fourth-order logarithmic kernel $(8\pi^{2}\tau)^{-1}\log(1/|x-y|)$, in the sense of a refined profile/quantisation theory analogous to \cite{raphael2014stability};
\item whether, in the absence of finite-time blow-up, infinite-time aggregation can occur, in analogy with the indirect-production critical-mass phenomenon of \cite{tao2017critical};
\item at the threshold $M=M_{\!*}$, whether mass quantisation holds.
\end{enumerate}
\end{problem}

A complete answer requires more than the energy identity: it requires a dynamical mechanism (virial identity, monotonicity formula, or sharp compactness) linking the fourth-order logarithmic attraction to spatial concentration.

\subsection{Memory criticality for MEP}
\label{subsec:MEP-criticality}

For MEP, the structural difficulty is that the drift cannot be replaced by a static operator: Remark~\ref{rem:not-static-reduction} forbids transferring the PES exponent $q_c=N/4$ via the frozen-time average alone. A proper threshold theory must operate at the level of the bilinear nonlinearity, in mixed space--time norms.

\begin{problem}[MEP memory criticality]\label{prob:MEP-criticality}
Develop a critical theory for the MEP system \eqref{eq:MEP-intro} based on the Volterra drift $\Vop u(t)=\int_0^t\nabla e^{-(t-s)A}A^{-1}u(s)\,ds$. Specifically, determine whether there exists a Banach space $X$ of densities satisfying:
\begin{enumerate}[label=(\roman*),leftmargin=2em]
\item $X$ is invariant (or critical) under the parabolic scaling of the density equation;
\item the Volterra estimate \emph{(}Theorem~\ref{thm:MEP-drift}\emph{)} closes the drift bound: $u\in X$ implies $\nabla c\in Y$ for a drift space $Y$ controlling $\Div(u\nabla c)$;
\item the bilinear map
\[
        (u,v)\longmapsto\int_0^t \nabla e^{(t-s)\Delta}\!\cdot\!\left[u(s)\int_0^s\nabla e^{-(s-\sigma)A}A^{-1}v(\sigma)\,d\sigma\right]ds
\]
is bounded on $X\times X\to X$;
\item the continuation criterion of Proposition~\ref{prop:MEP-continuation} can be replaced by a critical or subcritical condition in $X$;
\item the resulting threshold distinguishes whether MEP exhibits Keller--Segel-type criticality, PES-type averaged criticality, or a genuinely new memory-induced threshold.
\end{enumerate}
\end{problem}

The interaction between the near-diagonal Keller--Segel order $-1$ and the frozen-time averaged order $-3$ is the essential open question for MEP. We do not assert any specific critical exponent in advance of the resolution of Problem~\ref{prob:MEP-criticality}.

\subsection{Conditional statements}
\label{subsec:conditional-statements}

We close with two conditional statements that follow immediately from the structural analysis. Both record what is implied by, respectively, an affirmative resolution of Problem~\ref{prob:sharp-logHLS} and a choice of admissible space $X$ in Problem~\ref{prob:MEP-criticality}.

\begin{proposition}[Conditional PES lower bound]\label{prop:conditional-PES}
Let $N=4$, $M>0$, and suppose Problem~\ref{prob:sharp-logHLS} admits a solution: there exist constants $a(M,\tau,\Omega)$ and $C(M,\tau,\Omega)$ such that
\begin{equation}\label{eq:conditional-inequality}
        \int_\Omega u\,K_\tau u\,dx\le a(M,\tau,\Omega)\int_\Omega u\log u\,dx+C(M,\tau,\Omega)
\end{equation}
for all $u\in\Adm_M(\Omega)$, with $\tfrac{\chi}{2}a(M,\tau,\Omega)<1$. Then
\begin{equation}\label{eq:conditional-bound}
        \Free_{\mathrm{PES}}[u]\ge\Bigl(1-\tfrac{\chi}{2}a(M,\tau,\Omega)\Bigr)\!\int_\Omega u\log u\,dx-\tfrac{\chi}{2}C(M,\tau,\Omega),
\end{equation}
and consequently $\inf_{u\in\Adm_M(\Omega)}\Free_{\mathrm{PES}}[u]>-\infty$ (under the usual tightness assumption on $\R^{4}$).
\end{proposition}

\begin{proof}
Insert \eqref{eq:conditional-inequality} into the definition of $\Free_{\mathrm{PES}}$:
\[
        \Free_{\mathrm{PES}}[u]
        =\int u\log u-\tfrac{\chi}{2}\int u K_\tau u
        \ge\Bigl(1-\tfrac{\chi}{2}a(M,\tau,\Omega)\Bigr)\!\int u\log u-\tfrac{\chi}{2}C(M,\tau,\Omega).
\]
The coefficient is strictly positive by hypothesis, and $\int u\log u$ is bounded below on $\Adm_M$ at fixed mass (using Jensen's inequality on bounded domains, or moment-Jensen on $\R^{4}$).
\end{proof}

\begin{proposition}[Conditional MEP continuation]\label{prop:conditional-MEP}
Let $(u,c)$ be a classical MEP solution on $[0,T_{\max})$. Suppose there exists a mixed-norm space $X$ such that, for any $T<T_{\max}$,
\begin{equation}\label{eq:conditional-MEP-control}
        \|u\|_{X([0,T])}<\infty\quad\Longrightarrow\quad\|\nabla c\|_{L^{1}(0,T;W^{1,\infty})}<\infty,
\end{equation}
through the Volterra representation \eqref{eq:MEP-Duhamel}. If, in addition, $\|u\|_{L^{\infty}(0,T_{\max};L^{\infty})}<\infty$, then $T_{\max}=\infty$.
\end{proposition}

\begin{proof}
Combine the hypothesis with Proposition~\ref{prop:MEP-continuation}.
\end{proof}

\subsection{Conclusion}
\label{subsec:conclusion}

The mathematical analysis reveals that the physical assumption of instantaneous equilibration versus transient kinetic memory fundamentally alters the macroscopic pattern-formation thresholds. Specifically, we identify two distinct criticality classes for indirect signal-generation chemotaxis cascades:

\begin{enumerate}[label=(\roman*),leftmargin=2em]
\item \textbf{PES} is a \emph{static fourth-order elliptic} chemotaxis system. The eliminated interaction $c=K_\tau u$ is positive, self-adjoint, and fourth-order smoothing; the drift has principal order $-3$; the critical scaling is $u_\lambda(x,t)=\lambda^{4}u(\lambda x,\lambda^{2}t)$ with $q_c=N/4$; mass is critical in $N=4$; and the four-dimensional kernel has logarithmic local coefficient $1/(8\pi^{2}\tau)$. The variational candidate critical mass is
\[
        M_{\!*}=\frac{64\pi^{2}\tau}{\chi}.
\]
Promoting $M_{\!*}$ to a sharp threshold reduces to the sharp $K_\tau$-adapted Adams/log-HLS inequality of Problem~\ref{prob:sharp-logHLS}.

\item \textbf{MEP} is a \emph{Volterra-memory} chemotaxis system. Its signal law is $c(t)=e^{-tA}c_0+\int_0^t e^{-(t-s)A}A^{-1}u(s)\,ds$. Its near-diagonal drift has Keller--Segel order $-1$; its frozen-time average has fourth-order order $-3$. The two competing features prevent a static fourth-order elliptic reduction. The corresponding threshold problem is a genuinely mixed parabolic--Volterra problem (Problem~\ref{prob:MEP-criticality}).
\end{enumerate}

The fundamental conclusion is that \textit{PES and MEP belong to different criticality classes.} PES belongs to the fourth-order elliptic Adams/log-HLS class, while MEP belongs to a Volterra-memory class whose critical behaviour cannot be inferred from the PES scaling, nor from a static elliptic argument of any kind. Recognising this operator-theoretic separation is an essential first step for correctly formulating critical mass phenomena in realistic multiscale models, where the intermediate kinetics of microenvironments often retain critical temporal memory.


\begin{thebibliography}{10}

\bibitem{keller1970initiation}
Evelyn~F Keller and Lee~A Segel.
\newblock Initiation of slime mold aggregation viewed as an instability.
\newblock {\em Journal of theoretical biology}, 26(3):399--415, 1970.

\bibitem{keller1971model}
Evelyn~F Keller and Lee~A Segel.
\newblock Model for chemotaxis.
\newblock {\em Journal of theoretical biology}, 30(2):225--234, 1971.

\bibitem{beckner1993sharp}
William Beckner.
\newblock Sharp {S}obolev inequalities on the sphere and the {M}oser--{T}rudinger inequality.
\newblock {\em Annals of Mathematics}, 138(1):213--242, 1993.

\bibitem{carlen1992competing}
ERIC Carlen and Michael Loss.
\newblock Competing symmetries, the logarithmic {H}{L}{S} inequality and {O}nofri's inequality on sn.
\newblock {\em Geometric \& Functional Analysis GAFA}, 2(1):90--104, 1992.

\bibitem{blanchet2006two}
Adrien Blanchet, Jean Dolbeault, and Beno{\^\i}t Perthame.
\newblock Two-dimensional {K}eller--{S}egel model: {O}ptimal critical mass and qualitative properties of the solutions.
\newblock {\em Electronic Journal of Differential Equations (EJDE)[electronic only]}, 2006:Paper--No, 2006.

\bibitem{jager1992explosions}
Willi J{\"a}ger and Stephan Luckhaus.
\newblock On explosions of solutions to a system of partial differential equations modelling chemotaxis.
\newblock {\em Transactions of the american mathematical society}, 329(2):819--824, 1992.

\bibitem{nagai1995blow}
Toshitaka Nagai.
\newblock Blow-up of radially symmetric solutions to a chemotaxis system.
\newblock {\em Adv. Math. Sci. Appl.}, 5:581, 1995.

\bibitem{gajewski1998global}
Herbert Gajewski, Klaus Zacharias, and Konrad Gr{\"o}ger.
\newblock Global behaviour of a reaction-diffusion system modelling chemotaxis.
\newblock {\em Mathematische Nachrichten}, 195(1):77--114, 1998.

\bibitem{liu2025bidirectional}
Zonghao Liu, Louis~Shuo Wang, Jiguang Yu, Jilin Zhang, Erica Martel, and Shijia Li.
\newblock Bidirectional endothelial feedback drives {T}uring-vascular patterning and drug-resistance niches: a hybrid {P}{D}{E}-agent-based study.
\newblock {\em Bioengineering}, 12(10):1097, 2025.

\bibitem{horstmann20031970}
Dirk Horstmann.
\newblock From 1970 until present: the {K}eller--{S}egel model in chemotaxis and its consequences.
\newblock 2003.

\bibitem{fujie2017application}
Kentarou Fujie and Takasi Senba.
\newblock Application of an {A}dams type inequality to a two-chemical substances chemotaxis system.
\newblock {\em Journal of Differential Equations}, 263(1):88--148, 2017.

\bibitem{tao2017critical}
Youshan Tao and Michael Winkler.
\newblock Critical mass for infinite-time aggregation in a chemotaxis model with indirect signal production.
\newblock {\em Journal of the European Mathematical Society (EMS Publishing)}, 19(12), 2017.

\bibitem{wang2026algebraic}
Louis~Shuo Wang and Jiguang Yu.
\newblock Algebraic--spectral thresholds and discrete--continuous stability transfer in {L}eslie--{G}ower systems.
\newblock {\em Electronic Research Archive}, 34(1):251--290, 2026.

\bibitem{hillen2009user}
Thomas Hillen and Kevin~J Painter.
\newblock A user’s guide to {P}{D}{E} models for chemotaxis.
\newblock {\em Journal of mathematical biology}, 58(1):183--217, 2009.

\bibitem{bellomo2015toward}
Nicola Bellomo, Abdelghani Bellouquid, Youshan Tao, and Michael Winkler.
\newblock Toward a mathematical theory of {K}eller--{S}egel models of pattern formation in biological tissues.
\newblock {\em Mathematical Models and Methods in Applied Sciences}, 25(09):1663--1763, 2015.

\bibitem{wang2025analysis}
Louis~Shuo Wang, Jiguang Yu, Shijia Li, and Zonghao Liu.
\newblock Analysis and mean-field limit of a hybrid {P}{D}{E}--{A}{B}{M} modeling angiogenesis-regulated resistance evolution.
\newblock {\em Mathematics}, 13(17):2898, 2025.

\bibitem{winkler2016two}
Michael Winkler.
\newblock The two-dimensional {K}eller--{S}egel system with singular sensitivity and signal absorption: {G}lobal large-data solutions and their relaxation properties.
\newblock {\em Mathematical Models and Methods in Applied Sciences}, 26(05):987--1024, 2016.

\bibitem{blanchet2008infinite}
Adrien Blanchet, Jos{\'e}~A Carrillo, and Nader Masmoudi.
\newblock Infinite time aggregation for the critical {P}atlak--{K}eller--{S}egel model in $\mathbb r^2$.
\newblock {\em Communications on Pure and Applied Mathematics: A Journal Issued by the Courant Institute of Mathematical Sciences}, 61(10):1449--1481, 2008.

\bibitem{liang2025global}
Ye~Liang, Louis~Shuo Wang, Jiguang Yu, and Zonghao Liu.
\newblock Global well-posedness and stability of nonlocal damage-structured lineage model with feedback and dedifferentiation.
\newblock {\em Mathematics}, 13(22):3583, 2025.

\bibitem{raphael2014stability}
Pierre Rapha{\"e}l and R{\'e}mi Schweyer.
\newblock On the stability of critical chemotactic aggregation.
\newblock {\em Mathematische Annalen}, 359(1):267--377, 2014.

\bibitem{nagai1997application}
Toshitaka Nagai, Takasi Senba, and Kiyoshi Yoshida.
\newblock Application of the {T}rudinger-{M}oser inequality to a parabolic system of chemotaxis.
\newblock {\em Funkc. Ekvacioj}, 40:411--433, 1997.

\bibitem{xiang2025critical}
Zhaoyin Xiang and Lan Yang.
\newblock Critical mass for the {C}auchy problem of a chemotaxis model with indirect signal production mechanism.
\newblock {\em Journal of Evolution Equations}, 25(1):26, 2025.

\bibitem{yang2026critical}
Lan Yang.
\newblock Critical mass for a no-flux-{D}irichlet chemotaxis model with indirect signal production mechanism.
\newblock {\em Advances in Nonlinear Analysis}, 15(1):20250147, 2026.

\bibitem{laurenccot2024singular}
Philippe Lauren{\c{c}}ot and Christian Stinner.
\newblock Singular limit of a chemotaxis model with indirect signal production and phenotype switching.
\newblock {\em Nonlinearity}, 37(10):105007, 2024.

\bibitem{wang2025analysis1}
Louis~Shuo Wang and Jiguang Yu.
\newblock Analysis framework for stochastic predator--prey model with demographic noise.
\newblock {\em Results in Applied Mathematics}, 27:100621, 2025.

\bibitem{mao2025critical}
Xuan Mao, Meng Liu, and Yuxiang Li.
\newblock Critical mass for finite-time chemotactic collapse in the critical dimension via comparison.
\newblock {\em Nonlinearity}, 38(9):095026, 2025.

\bibitem{mao2025finite}
Xuan Mao, Meng Liu, and Yuxiang Li.
\newblock Finite-time blowup in a fully parabolic chemotaxis model involving indirect signal production.
\newblock {\em arXiv preprint arXiv:2503.12439}, 2025.

\bibitem{mao2025finite1}
Xuan Mao and Yuxiang Li.
\newblock Finite-time blowup in a parabolic-parabolic-elliptic chemotaxis model involving indirect signal production.
\newblock {\em Applied Mathematics \& Optimization}, 92(1):10, 2025.

\bibitem{wang2026damage}
Louis~Shuo Wang, Jiguang Yu, and Zonghao Liu.
\newblock A damage-structured {P}{D}{E} model of stem cell hierarchies: The dual role of dedifferentiation in tissue homeostasis and aging.
\newblock {\em Plos one}, 21(2):e0335163, 2026.

\bibitem{soga2026sharp}
Yuri Soga.
\newblock A sharp criterion and complete classification of global-in-time solutions and finite time blow-up of solutions to a chemotaxis system in supercritical dimensions.
\newblock {\em arXiv preprint arXiv:2601.15990}, 2026.

\bibitem{winkler2022family}
Michael Winkler.
\newblock A family of mass-critical {K}eller--{S}egel systems.
\newblock {\em Proceedings of the London Mathematical Society}, 124(2):133--181, 2022.

\bibitem{yu2026pattern}
Jiguang Yu, Louis~Shuo Wang, Zonghao Liu, and Jingfeng Liu.
\newblock Pattern suppression and recovery under one-way versus two-way chemotactic coupling in hybrid partial differential equation--ordinary differential equation models.
\newblock {\em Transport Phenomena}, 2026.

\bibitem{chen2024negligibility}
Yuanlin Chen and Tian Xiang.
\newblock Negligibility of haptotaxis on global dynamics in a chemotaxis-haptotaxis system with indirect signal production.
\newblock {\em Journal of Differential Equations}, 409:1--48, 2024.

\bibitem{pazy2012semigroups}
Amnon Pazy.
\newblock {\em Semigroups of linear operators and applications to partial differential equations}.
\newblock Springer Science \& Business Media, 2012.

\bibitem{henry2006geometric}
Daniel Henry.
\newblock {\em Geometric theory of semilinear parabolic equations}.
\newblock Springer, 2006.

\bibitem{lunardi2012analytic}
Alessandra Lunardi.
\newblock {\em Analytic semigroups and optimal regularity in parabolic problems}.
\newblock Springer Science \& Business Media, 2012.

\bibitem{wang2026breakdown}
Louis~Shuo Wang, Jiguang Yu, Ye~Liang, and Jilin Zhang.
\newblock The breakdown of linear quasi-cycles: {D}emographic noise and absorbing boundaries in finite predator--prey systems.
\newblock {\em Electronic Research Archive}, 34(6):4248--4289, 2026.

\bibitem{amann1995linear}
Herbert Amann et~al.
\newblock {\em Linear and quasilinear parabolic problems}, volume~1.
\newblock Springer, 1995.

\bibitem{pruss2016moving}
Jan Pr{\"u}ss and Gieri Simonett.
\newblock {\em Moving interfaces and quasilinear parabolic evolution equations}, volume 105.
\newblock Springer, 2016.

\bibitem{shi2021spatial}
Qingyan Shi, Junping Shi, and Hao Wang.
\newblock Spatial movement with distributed memory.
\newblock {\em Journal of Mathematical Biology}, 82(4), 2021.

\bibitem{bui2026parabolic}
Le~Trong~Thanh Bui, Thi Kim~Loan Huynh, Bao~Quoc Tang, and Bao-Ngoc Tran.
\newblock Parabolic-elliptic and indirect-direct simplifications in chemotaxis systems driven by indirect signalling: Ltt bui et al.
\newblock {\em Calculus of Variations and Partial Differential Equations}, 65(3):76, 2026.

\bibitem{wang2025multi}
Zixin Wang, Danqing Wang, and Jiguang Yu.
\newblock Multi-strategy hybrid improved intelligent algorithm for solving {U}{A}{V}-{M}{T}{S}{P}.
\newblock {\em Information Technology and Control}, 54(2):413--438, 2025.

\bibitem{liu2025global}
Qian Liu and Dan Li.
\newblock Global existence for the {C}auchy problem of the parabolic--parabolic--{O}{D}{E} chemotaxis model with indirect signal production on the plane.
\newblock {\em Mathematics}, 13(16):2624, 2025.

\bibitem{nagai2000chemotactic}
Toshitaka Nagai, Takasi Senba, and Takashi Suzuki.
\newblock Chemotactic collapse in a parabolic system of mathematical biology.
\newblock {\em Hiroshima Mathematical Journal}, 30(3):463--497, 2000.

\bibitem{gao2022rolling}
Yuansheng Gao, Lei Li, and Jiguang Yu.
\newblock Rolling prediction model of closing price based on {E}{E}{M}{D} data noise reduction and {H}{G}{S}-{D}{E}{L}{M}.
\newblock In {\em 2022 International Conference on Data Analytics, Computing and Artificial Intelligence (ICDACAI)}, pages 255--260. IEEE, 2022.

\bibitem{mao2024dirac}
Xuan Mao and Yuxiang Li.
\newblock Dirac-type aggregation with full mass in a chemotaxis model.
\newblock {\em Discrete and Continuous Dynamical Systems-S}, 17(4):1513--1528, 2024.

\bibitem{adams1988sharp}
David~R Adams.
\newblock A sharp inequality of {J}. {M}oser for higher order derivatives.
\newblock {\em Annals of Mathematics}, 128(2):385--398, 1988.

\bibitem{adams2003sobolev}
Robert~A Adams and John~JF Fournier.
\newblock {\em Sobolev spaces}, volume 140.
\newblock Elsevier, 2003.

\bibitem{hebey2000nonlinear}
Emmanuel Hebey.
\newblock {\em Nonlinear analysis on manifolds: {S}obolev spaces and inequalities}, volume~5.
\newblock American Mathematical Soc., 2000.

\bibitem{yu2026beyond}
J.~Yu and L.~S. Wang.
\newblock Beyond diagonal noise: {A} better predator-prey modeling framework with cross-covariance.
\newblock {\em PLoS One}, 21(5):e0350127, 2026.

\bibitem{lin1998classification}
C-S Lin.
\newblock A classification of solutions of a conformally invariant fourth order equation in $\mathbb{R}^n$.
\newblock {\em Commentarii Mathematici Helvetici}, 73(2):206--231, 1998.

\bibitem{martinazzi2009classification}
Luca Martinazzi.
\newblock Classification of solutions to the higher order {L}iouville’s equation on $\mathbb{R}^{2m}$.
\newblock {\em Mathematische Zeitschrift}, 263(2):307--329, 2009.

\bibitem{lieb2001analysis}
Elliott~H Lieb and Michael Loss.
\newblock {\em Analysis}, volume~14.
\newblock American Mathematical Soc., 2001.

\bibitem{ruf2013sharp}
Bernhard Ruf and Federica Sani.
\newblock Sharp {A}dams-type inequalities in $\mathbb R^n$.
\newblock {\em Transactions of the American Mathematical Society}, 365(2):645--670, 2013.

\bibitem{laurenccot2017finite}
Philippe Lauren{\c{c}}ot and Noriko Mizoguchi.
\newblock Finite time blowup for the parabolic--parabolic {K}eller--{S}egel system with critical diffusion.
\newblock In {\em Annales de l'Institut Henri Poincar{\'e} C, Analyse non lin{\'e}aire}, volume~34, pages 197--220. Elsevier, 2017.

\bibitem{senba2002weak}
Takasi Senba and Takashi Suzuki.
\newblock Weak solutions to a parabolic--elliptic system of chemotaxis.
\newblock {\em Journal of Functional Analysis}, 191(1):17--51, 2002.

\bibitem{yu2026rigorous}
Jiguang Yu, Louis~Shuo Wang, and Ye~Liang.
\newblock Rigorous analysis of a nonlocal transport–renewal system for physiologically structured populations.
\newblock {\em Mathematical Methods in the Applied Sciences}, 2026.

\bibitem{fujie2019blowup}
Kentarou Fujie and Takasi Senba.
\newblock Blowup of solutions to a two-chemical substances chemotaxis system in the critical dimension.
\newblock {\em Journal of Differential Equations}, 266(2-3):942--976, 2019.

\bibitem{hosono2025global}
Tatsuya Hosono and Philippe Lauren{\c{c}}ot.
\newblock Global existence and boundedness of solutions to a fully parabolic chemotaxis system with indirect signal production in $\mathbb{R}^4$.
\newblock {\em Journal of Differential Equations}, 416:2085--2133, 2025.

\bibitem{triebel1992theory}
H~Triebel.
\newblock Theory of function spaces {I}{I}.
\newblock {\em Birkh{\"a}user Verlag}, 1992.

\bibitem{stein1970singular}
Elias~M Stein.
\newblock {\em Singular integrals and differentiability properties of functions}.
\newblock Number~30. Princeton university press, 1970.

\bibitem{taylor2010pseudodifferential}
Michael~E Taylor.
\newblock Pseudodifferential operators.
\newblock In {\em Partial Differential Equations {I}{I}: Qualitative Studies of Linear Equations}, pages 1--90. Springer, 2010.

\bibitem{jordan1998variational}
Richard Jordan, David Kinderlehrer, and Felix Otto.
\newblock The variational formulation of the {F}okker--{P}lanck equation.
\newblock {\em SIAM journal on mathematical analysis}, 29(1):1--17, 1998.

\bibitem{ambrosio2005gradient}
Luigi Ambrosio, Nicola Gigli, and Giuseppe Savar{\'e}.
\newblock {\em Gradient flows: in metric spaces and in the space of probability measures}.
\newblock Springer, 2005.

\bibitem{ladyzhenskaia1968linear}
Olga~Aleksandrovna Ladyzhenskaia, Vsevolod~Alekseevich Solonnikov, and Nina~N Uraltseva.
\newblock {\em Linear and quasi-linear equations of parabolic type}, volume~23.
\newblock American Mathematical Soc., 1968.

\bibitem{lam2012sharp}
Nguyen Lam and Guozhen Lu.
\newblock Sharp {A}dams type inequalities in {S}obolev spaces $W^{m,\frac{n}{m}}(\mathbb{R}^n)$ for arbitrary integer $m$.
\newblock {\em Journal of Differential Equations}, 253(4):1143--1171, 2012.

\end{thebibliography}
\end{document}